\documentclass{compositio}        
\usepackage[utf8]{inputenc}       
\usepackage[T1]{fontenc}
\usepackage[small,it]{caption}
\usepackage[english]{babel}
\usepackage{url}
\usepackage{graphicx}                         
\usepackage{amsmath}
\usepackage{mathtools}
\usepackage{amssymb} 
\usepackage{enumerate}
\usepackage{tikz-cd}
\usepackage{euscript}

\theoremstyle{definition}
\newtheorem{para}{§\hspace{-.15em}}[section]
\swapnumbers
\theoremstyle{plain}
\newtheorem{thm}[para]{Theorem}
\newtheorem{lem}[para]{Lemma}
\newtheorem{prop}[para]{Proposition}
\newtheorem{cor}[para]{Corollary}
\theoremstyle{definition}
\newtheorem{defn}[para]{Definition}
\newtheorem{examplex}[para]{Example}
\newtheorem{rem}[para]{Remark}

\newcommand{\CF}{\EuScript{CF}}
\newcommand{\CD}{\EuScript{CD}}

\newcommand{\I}{\mathsf I}
\newcommand{\Fin}{\mathsf{Fin}}
\newcommand{\nerve}[1]{\vert \kern-.1em\vert#1\vert \kern-.1em\vert}
\newcommand{\nerveu}[1]{\lceil \kern-.27em\lceil \kern-.12em #1\kern-.12em\rceil\kern-.27em\rceil}
\newcommand{\nerved}[1]{\lfloor \kern-.27em\lfloor \kern-.12em #1\kern-.12em\rfloor\kern-.27em\rfloor}
\newcommand{\nerveud}[1]{\mathrlap{\lfloor}\lceil \kern-.27em\mathrlap{\lfloor}\lceil \kern-.12em #1\kern-.12em\mathrlap{\rfloor}\rceil\kern-.27em\mathrlap{\rfloor}\rceil}

\newcommand{\field}[1]{\ensuremath{\mathbf{#1}}}
\newcommand{\Q}{\ensuremath{\field{Q}}}        
\newcommand{\Z}{\ensuremath{\field{Z}}} 
\newcommand{\R}{\ensuremath{\field{R}}} 
\DeclareMathOperator*{\hocolim}{hocolim}
\title{Cohomology of generalized configuration spaces}

\author{Dan Petersen}
\thanks{The author gratefully acknowledges support by ERC-2017-STG 759082.}
\email{dan.petersen@math.su.se}
\address{Matematiska Institutionen\\ Stockholms Universitet\\ 106 91 Stockholm\\ Sweden}

\begin{document} 
	
	\begin{abstract}Let $X$ be a topological space. We consider certain generalized configuration spaces of points on $X$, obtained from the cartesian product $X^n$ by removing some intersections of diagonals. We give a systematic framework for  studying the cohomology of such spaces using what we call ``tcdga models'' for the cochains on $X$. We prove the following theorem: suppose that $X$ is a ``nice'' topological space, $R$ is any commutative ring, $H^\bullet_c(X,R)\to H^\bullet(X,R)$ is the zero map, and that $H^\bullet_c(X,R)$ is a projective $R$-module. Then the compact support cohomology of any generalized configuration space of points on $X$ depends only on the graded $R$-module $H^\bullet_c(X,R)$. This generalizes a theorem of Arabia. \end{abstract}
 \maketitle

\section{Introduction}

\begin{para}\label{firstsection}
	If $X$ is a Hausdorff space, let $F(X,n)$ denote the \emph{configuration space} of $n$ distinct ordered points on $X$. A basic question in the study of configuration spaces is the following: 
	\[\text{\emph{How do the (co)homology groups of $F(X,n)$ depend on the (co)homology groups of $X$?}}\]
	This question has a slightly nicer answer if one considers \emph{cohomology with compact support} instead of the usual cohomology. This can be seen already for $n=2$, in which case there is a Gysin long exact sequence
	$$ \ldots \to H^k_c(F(X,2)) \to H^k_c(X^2) \to H^k_c(X) \to H^{k+1}_c(F(X,2)) \to \ldots $$
	associated to the inclusion of $F(X,2)$ as an open subspace of $X^2$. If we work with field coefficients, then $H^k_c(X^2) \cong \bigoplus_{p+q=k}H^p_c(X) \otimes H^q_c(X)$ and the map $H^k_c(X^2) \to H^k_c(X)$ is given by multiplication in the cohomology ring $H^\bullet_c(X)$; consequently, the compactly supported cohomology groups of $F(X,2)$ are completely determined by the compactly supported cohomology of $X$, \emph{with its ring structure}. No such simple statement is true for the usual cohomology or homology.\end{para}
	\begin{para}
	For $n > 2$ it is no longer true that $H^\bullet_c(F(X,n))$ depends only on the ring structure on $H^\bullet_c(X)$. However, as a consequence of the results proved in this paper, we will see that if $X$ is any locally compact Hausdorff space, and $R$ is any ring,\footnote{From now until forever, all rings are assumed to be commutative.} then the graded $\mathbb S_n$-module $H^\bullet_c(F(X,n),R)$ depends only on the quasi-isomorphism class of the \emph{$\mathbb E_\infty$-algebra} given by the \emph{compactly supported cochains} $C_c^\bullet(X,R)$.\footnote{If $R=\Z$ and the one-point compactification $X^\ast$ of $X$ is finite type nilpotent, this is an easy consequence of Mandell's theorem \cite{mandell}, at least modulo point-set issues. Indeed, we may identify $C_c^\bullet(X,R)$ with the reduced cochains of $X^\ast$, so that the $\mathbb E_\infty$-structure on the compactly supported cochains determines the homotopy type of $X^\ast$. Then $F(X,n)^\ast$ is given by collapsing the big diagonal inside the $n$-fold smash product of $X^\ast$ with itself, which is clearly homotopy invariant.} In fact, the full $\mathbb E_\infty$-structure is not needed: there is a forgetful functor from $\mathbb E_\infty$-algebras to \emph{twisted commutative dg algebras}, and one only needs to know $C_c^\bullet(X,R)$ as a twisted commutative dg algebra. If we are not interested in the $\mathbb S_n$-module structure on $H^\bullet_c(F(X,n),R)$ one needs even less information: we may consider $C_c^\bullet(X,R)$ as what's known as a \emph{commutative shuffle dg algebra}.	 With field coefficients one can show using homological perturbation theory that $H^\bullet_c(F(X,n))$ depends only on the ring $H^\bullet_c(X)$ \emph{together with its higher Massey products} --- in fact, one only needs to know the $m$-fold Massey products for $m \leq n$, which recovers what we said above for $n=2$. When $X$ is a compact oriented manifold, this is a theorem of Baranovsky--Sazdanovic \cite{baranovskysazdanovic}.\end{para}

\begin{para}
	The prototype of our first main theorem is a result of Bendersky and Gitler \cite{benderskygitler}. They define for any integer $n \geq 1$ an explicit and combinatorially defined functor $\Lambda(n,-)$ from commutative dg algebras (cdga's) to $\mathbb S_n$-equivariant cochain complexes. If $A \to A'$ is a cdga quasi-isomorphism, then $\Lambda(n,A) \to \Lambda(n,A')$ is a quasi-isomorphism. If $A$ is a cdga model for the cochains $C^\bullet(X,\Q)$ of a space $X$, then this cochain complex computes the following relative cohomology group:
	$$ H(\Lambda(n,A)) \cong H^\bullet(X^n,D(X,n); \Q),$$
	where $D(X,n) = X^n \setminus F(X,n)$ is the complement of the configuration space (i.e.\ the ``big diagonal'').
	
	Our first main theorem is the construction of two functors $\CF(U,A)$ and $\CD(U,A)$, where $U$ is an upwards closed subset of the $n$th partition lattice $\Pi_n$ and $A$ is a \emph{twisted} commutative dg algebra. When $U = \Pi_n\setminus \{\hat 0\}$ and $A$ is a cdga, we have $\CF(U,A) \simeq \Lambda(n,A)$. The functors $\CF$ and $\CD$ improve upon the construction of Bendersky and Gitler in three directions:
	\begin{enumerate}[(A)]
		\item The construction of Bendersky--Gitler works only for cohomology with coefficients in $\Q$. By replacing commutative dg algebras with twisted commutative dg algebras, we can treat in a uniform manner coefficients in an arbitrary ring (or in fact any sheaf or complex of sheaves).
		\item The construction here works for more general ``configuration-like'' spaces: to an upwards closed subset $U \subset \Pi_n$ we can associate an open subset of $X^n$ obtained by removing a family of intersections of diagonals, and any such open subset arises from some upwards closed $U$. Our functors $\CF(U,-)$ and $\CD(U,-)$ compute the cohomology of the resulting generalized configuration space, and the cohomology of the ``discriminant'', respectively.
		\item We give a construction that works equally well for computing the \emph{compactly supported cohomology} of the configuration space of points on $X$; that this should be possible is not at all obvious from how the functor $\Lambda(n,-)$ is constructed by Bendersky--Gitler. In particular, if $X$ is an oriented manifold, we obtain a complex which computes the homology of $F(X,n)$ (since Poincar\'e duality identifies homology and compact support cohomology).
		
	\end{enumerate}
	Let us, in turn, comment on each of these points in some more detail. 
\end{para}

\subsection*{A. Arbitrary coefficients.}

\begin{para} In Bendersky and Gitler's definition of $\Lambda(n,A)$, commutativity of $A$ is used crucially in order to verify the equation $d^2 = 0$. This means that their results can only be used over a field of characteristic zero; otherwise, there will practically never exist a strictly commutative dg algebra model for the cochains on a space. For this reason, Baranovsky and Sazdanovic \cite{baranovskysazdanovic} remarked that it would be interesting to understand if the construction of $\Lambda(n,A)$ could be modified somehow to allow $A$ to be an $\mathbb E_\infty$-algebra. But the higher coherence conditions in an $\mathbb E_\infty$-algebra are unwieldy, no matter what model of $\mathbb E_\infty$-operad one chooses, and it is far from immediate how the functor $\Lambda$ should be modified.  	\end{para}

\begin{para}\label{I-alg}It was proven by Sagave and Schlichtkrull \cite{sagaveschlichtkrull} that $\mathbb E_\infty$-algebras, which are only commutative up to coherent homotopy, can be faithfully modeled by \emph{strictly commutative objects in a larger category}, in which the higher coherences are in a sense built into the objects of the category themselves. Specifically, they introduced the notion of an $\mathsf I$-algebra; an $\mathsf I$-algebra is a commutative monoid in the category of functors from finite sets and injections to cochain complexes. Many readers of this paper are perhaps familiar with the literature on representation stability --- in the usual lingo of representation stability, an $\mathsf I$-algebra is just a gadget which is simultaneously an $\mathsf {FI}$-module and a twisted commutative algebra. For a suitable model structure on $\mathsf I$-algebras, this category is Quillen equivalent to the category of $\mathbb E_\infty$-algebras. 
\end{para}
\begin{para}
	It turns out that Bendersky--Gitler's functor $\Lambda$, and our generalized versions $\CF$ and $\CD$, can be readily modified to make sense also for commutative dg $\mathsf I$-algebras. In fact, one does not even need the injections, so the full $\mathbb E_\infty$-structure is not needed --- the functor is well defined already on the category of twisted commutative dg algebras (tcdga's). We shall think of the category of tcdga's as a useful enlargement of the category of cdga's; small enough that its objects are specified by a manageable amount of data, and large enough to contain all examples of interest.
\end{para}

\begin{para}
	The passage from cdga's to tcdga's not only allows us to pass from $\Q$-coefficients to $\Z$-coefficients; it is also needed if we wish to deal with cohomology with twisted coefficients. If $\EuScript F$ is a complex of sheaves on $X$, then we define tcdga's $R\Gamma^\otimes(X,\EuScript F)$ and $R\Gamma^\otimes_c(X,\EuScript F)$ whose cohomologies are given by
	$$ \bigoplus_{n \geq 1} H^\bullet(X,\EuScript F^{\otimes n}) \qquad \text{and} \qquad \bigoplus_{n \geq 1} H^\bullet_c(X,\EuScript F^{\otimes n})$$
	with multiplication given by $H^k(X,\EuScript F^{\otimes n}) \otimes H^k(X,\EuScript F^{\otimes m}) \to H^{k+l}(X,\EuScript F^{\otimes (n+m)})$. Evaluating our functor $\CF$ on these tcdga's gives a cochain complex calculating the cohomologies of
	$$ H^\bullet (X^n, D(X,U); \EuScript F^{\boxtimes n}) \qquad \text{and} \qquad H^\bullet_c(F(X,U),\EuScript F^{\boxtimes n}),$$
	respectively. The additional generality of being allowed to use arbitrary coefficients is in fact useful. For example, we can recover a construction of Knudsen \cite{knudsenconfiguration} computing the rational homology of \emph{unordered} configuration spaces of points on a manifold $X$ by taking for $\EuScript F$ the orientation local system on $X$ and taking $\mathbb S_n$-invariants. More generally, if we let $\EuScript F = \mathbb D \Z_X$ be the dualizing complex on the space $X$, then $H^\bullet_c(F(X,U),\EuScript F^{\boxtimes n})$ is the integer homology of $F(X,U)$ and we get a spectral sequence calculating the homology of the configuration space of points on an arbitrary space. In general it is a hard problem to compute $\mathbb D\Z_X$ for non-manifold $X$, but even partial information can be used to obtain qualitative results on the cohomology of configuration spaces. For example, Tosteson \cite{tostesonconfiguration} shows, using a spectral sequence equivalent to the one here, that if $\EuScript H^i(\mathbb D\Z_X)=0$  for $i < 2$ ($X$ is ``locally $\geq 2$-dimensional'') then the configuration spaces of points on $X$ satisfy representation stability.
\end{para}

\begin{para}
	Our constructions of the tcdga's $R\Gamma^\otimes(X,\EuScript F)$ and $R\Gamma^\otimes_c(X,\EuScript F)$ is quite simple, but we hope that it can be of independent interest. The case when $\EuScript F$ is a constant sheaf is particularly interesting --- in this case one can give a natural $\mathsf {FI}$-module structure on these tcdga's as well, so that we obtain an $\mathsf I$-algebra in the sense of \S \ref{I-alg}. Under the Quillen equivalence between $\mathsf I$-algebras and $\mathbb E_\infty$-algebras, these $\mathsf I$-algebras are equivalent to the cochains (resp.\ compactly supported cochains) on $X$, with their $\mathbb E_\infty$-algebra structure (so we recover in a slightly unusual way the $\mathbb E_\infty$-structure on cochains). A different explicit $\mathsf I$-algebra structure for the cochains on a space was very recently constructed by Richter and Sagave \cite{richtersagave}; our construction gives a sheaf-theoretic alternative to theirs.
\end{para}

\subsection*{B. Compact support cohomology}

\begin{para} Let us momentarily consider the simplest case of rational coefficients and the classical configuration spaces $F(X,n)$, where our construction is equivalent to the functor $\Lambda(n,-)$ of Bendersky and Gitler. As we said above, if $A$ is a cdga model for the cochains $C^\bullet(X,\Q)$, then $\Lambda(n,A)$ computes the relative cohomology $H^\bullet(X^n,D(X,n); \Q)$. A consequence of what we prove here is that if $A$ is instead a cdga model for the \emph{compactly supported cochains} $C^\bullet_c(X,\Q)$, then we have an isomorphism
	$$ H(\Lambda(n,A)) \cong H^\bullet_c(F(X,n),\Q).$$
	In other words, the \emph{exact same functor} will, when we plug in the compactly supported cochains of $X$, calculate the compact support cohomology of the configuration space itself. This is not clear from Bendersky and Gitler's construction of the functor $\Lambda$. The same holds integrally if we choose a tcdga model for $C_c^\bullet(X,\Z)$, rather than a cdga model (which will practically never exist).
	\end{para}
	\begin{para}
		
	This is particularly interesting when $X$ is an oriented $d$-manifold. In this case we have a Poincar\'e duality isomorphism
	$$ H^k_c(F(X,n),\Z) \cong H_{nd-k}(F(X,n),\Z)$$
	between homology and cohomology, and we obtain a canonical cochain complex computing the homology of $F(X,n)$, depending only on a tcdga model for $C^\bullet_c(X,\Z)$. After giving this cochain complex a natural filtration one obtains a spectral sequence whose first page and first differential depend only on the ring $H^\bullet_c(X,\Z)$; this spectral sequence is exactly Poincar\'e dual to the familiar spectral sequence of Cohen--Taylor-Totaro-K{\v{r}}{\'{\i}}{\v{z}} \cite{cohentaylor,totaro,krizconfiguration}. This recovers some familiar statements: for example, the first differential in the Cohen--Taylor spectral sequence is given by the Gysin map $\Delta_! \colon H^\bullet(X) \to H^{\bullet+d}(X^2)$, which is exactly Poincar\'e dual to the cup product in compact support cohomology. However, the fact that the higher differentials depend only on the cochain algebra $C^\bullet_c(X)$ appears to be new. For example, we see that if the algebra $C^\bullet_c(X,\Q)$ is formal, then the Cohen--Taylor spectral sequence degenerates rationally after the first differential; again, this observation seems to be new. 
	\end{para}

\subsection*{C. Generalized configuration spaces}

\begin{para} 	Let $\Pi_n$ denote the \emph{partition lattice}, i.e.\ the partially ordered set of all partitions of the set $\{1,\ldots,n\}$, ordered by refinement. Each element $T \in \Pi_n$ corresponds to a locally closed subset $X(T) \subset X^n$:
	$$ X(T) = \{(x_1,\ldots,x_n) \in X^n : x_i = x_j \iff i \text{ and } j \text{ are in the same block of $T$}\}.$$
	If $U \subset \Pi_n$ is upwards closed, then we may define
	$$ D(X,U) = \bigcup_{T \in U} X(T), \qquad \qquad F(X,U) = X^n \setminus D(X,U).$$
	Then $F(X,U)$ is a ``generalized configuration space'' of points on $X$; if $U$ consists of all elements of $\Pi_n$ except the bottom element $\hat 0$, then $F(X,U) = F(X,n)$, and in general we get an intermediate open subset between $F(X,n)$ and $X^n$. 
	\end{para}
	\begin{para}We define functors $\CF(U,A)$ and $\CD(U,A)$ such that if $U = \Pi_n \setminus \{\hat 0\}$ and $A$ is a cdga, then $\CF(U,A) \simeq \Lambda(n,A)$. If $A$ is a tcdga model for the cochains (respectively, the compactly supported cochains) on $X$ then $\CF(U,A)$ computes the cohomology of $H^\bullet(X^n,D(X,U);\Z)$ (respectively, $H^\bullet_c(F(X,U),\Z)$. Similarly $\CD(U,A)$ computes the cohomology of $H^\bullet(D(X,U),\Z)$ (resp.\ $H^\bullet_c(D(X,U),\Z)$. 	Thus the results of this paper provide a uniform way of studying the cohomology and compact support cohomology of spaces of the form $F(X,U)$ and $D(X,U)$, for any upwards closed $U \subseteq \Pi_n$.
	\end{para}
	
	\begin{para}When $X$ is a Euclidean space $\R^d$, then $F(X,U)$ is the complement of a ``hypergraph arrangement'' (also called a diagonal arrangment) and our results recover the Goresky--MacPherson formula in this case (as well as its equivariant version).  For a general space $X$ we obtain a spectral sequence computing $H^\bullet_c(F(X,U),\Z)$ whose first page depends only on the cohomology of $X$ and the cohomology of the lower intervals in the partially ordered set $U$. (A more general form of this spectral sequence was described previously in \cite{spectralsequencestratification}.) The cohomologies of such lower intervals have been studied intensely, precisely because they can be interpreted (via the Goresky--MacPherson formula) as computing the cohomology of hypergraph arrangements in Euclidean space. Our results thus show that whenever we can compute the cohomologies of such lower intervals for some poset $U$, we can write down an explicit spectral sequence computing the compact support cohomology of $F(X,U)$, whose first page depends only on the ring $H^\bullet_c(X,\Z)$. Cases where the cohomologies of such lower intervals are known include $k$-equals arrangements \cite{bjornerwelker} and large classes of orbit arrangements \cite{kozlovgeneral}.
	\end{para}
	
	\begin{para}A remark is that Totaro and K{\v{r}}{\'{\i}}{\v{z}} both note that the structure of the Cohen--Taylor spectral sequence for an oriented $d$-manifold $X$ depends on knowing the cohomology of $F(\R^d,n)$, and that this should reflect the locally Euclidean structure of $X$. It is interesting that we see the exact same phenomenon here --- for an arbitrary space $X$, the combinatorial structure of the spectral sequence computing $H^\bullet_c(F(X,U),\Z)$ depends only on knowing the cohomology of the corresponding hypergraph arrangment in a Euclidean space --- even though there is no ``locally Euclidean'' structure in sight. 
\end{para}

\begin{para}
Let us now state the theorem, which is a composite of Theorems \ref{mainthm} and \ref{sseq-thm}. The comparison with the Cohen--Taylor spectral sequence appears in \S\S \ref{comparisonstart}--\ref{comparisonend}. The two functors $\CF$ and $\CD$ are defined in Definition \ref{defn-cf-cd}. The symbols $\nerveu{-}$ and $\nerveud{-}$ denote two variants of the order complex of a partially ordered set, which can be found in \S \ref{variant-of-order-cpx}. 
\end{para}

\begin{thm}[First main theorem]\label{firstmainthm}There exist explicit functors $\CF(U,A)$ and $\CD(U,A)$, where $A$ is a tcdga over a ring $R$ and $U \subset \Pi_n$ is an upwards closed subset, to cochain complexes over $R$. If $G \subset \mathbb S_n$ preserves $U$ then $G$ acts on $\CF(U,A)$. Let $X$ be a locally contractible paracompact Hausdorff space, $\EuScript F$ a bounded below complex of sheaves of $R$-modules on $X$. There are equivariant quasi-isomorphisms of cochain complexes
	\begin{align*}
	\CF(U,R\Gamma_c^\otimes(X,\EuScript F)) & \simeq C^\bullet_c(F(X,U),\EuScript F^{\boxtimes n}) \\
	\CF(U,R\Gamma^\otimes(X,\EuScript F)) & \simeq  C^\bullet(X^n, D(X,U);\EuScript F^{\boxtimes n}) \\
	\CD(U,R\Gamma_c^\otimes(X,\EuScript F)) & \simeq C_c^\bullet(D(X,U),\EuScript F^{\boxtimes n}) \\
	\CD(U,R\Gamma^\otimes(X,\EuScript F)) & \simeq  C^\bullet(D(X,U),\EuScript F^{\boxtimes n})
	\end{align*} which are natural in $X$, $\EuScript F$ and $U$. There are natural filtrations on $\CF(U,A)$ and $\CD(U,A)$ such that we get spectral sequences  
	$$E_1^{pq} = \bigoplus_{\substack{T \in J_U\\ \vert T \vert = p}} \widetilde H^{p+q}( \nerveu{ J_U^{\preceq T}}; A(T)) \implies H^{p+q}(\CD(U,A))$$
	and
	$$ E_1^{pq} = \bigoplus_{\substack{T \in J_{U_0}\\ \vert T \vert = p}} \widetilde H^{p+q}(\nerveud{ J_{U_0}^{\preceq T}}; A(T)) \implies H^{p+q}(\CF(U,A)),$$
	where $J_U \subseteq U$ is the subposet consisting of all joins of minimal elements of $U$, $J_{U_0} = J_U \cup \{\hat 0\}$. The resulting spectral sequence is Poincar\'e dual to the Cohen--Taylor spectral sequence when $A = R\Gamma_c^\otimes(X,\Z)$, $X$ is an oriented manifold, and $U = \Pi_n \setminus \{\hat 0\}$. \end{thm}

\begin{para} When $U = \Pi_n \setminus \{\hat 0\}$ consists of the whole partition lattice minus its bottom element --- this is the case in which the functor $\CF(U,-)$ computes the cohomology of the usual configuration space of points --- then the cochain complex  $\CF(U,-)$  can be given a somewhat unexpected interpretation in terms of operadic cohomology. Specifically, 
	we show in the final section of the paper that when $U = \Pi_n \setminus \{\hat 0\}$ then the cochain complex  $\CF(U,-)$ may be identified with the Chevalley--Eilenberg chains computing Lie algebra homology of a certain \emph{twisted} Lie algebra (a left module over the Lie operad). The precise statement is that $\CF(U,A)$ computes the Lie algebra homology of $\mathsf SA \otimes_H \mathsf{Lie}$ (the Hadamard tensor product of the suspension of $A$ with the operad $\mathsf{Lie}$, considered as a left module over itself). If we take $X$ to be a manifold, and $A = R\Gamma_c^\otimes(X,\mathbb D \Q)$, then we recover theorems of Knudsen \cite{knudsenconfiguration} and H\^o \cite{hoconfiguration} upon taking $\mathbb S_n$-invariants\footnote{Knudsen and H\^o both work with unordered configuration spaces, which is why we need to take $\mathbb S_n$-invariants to recover their results. However, the arguments they use can be modified to obtain an interpretation of the homology of ordered configuration spaces in terms of Lie algebra homology as well, provided that one replaces Lie algebras with twisted Lie algebras throughout; in this case their results coincide with the ones obtained here.} in this result. Here $\mathbb D \Q$ denotes the rational dualizing complex of $X$, i.e.\ the orientation sheaf with $\Q$-coefficients placed in degree $\dim X$. We also indicate how this result can be understood in terms of the Goodwillie calculus. That said, it is not clear to me what the precise relationship is between what is done in this paper, and methods such as higher Hochschild homology or manifold calculus. In particular, the recent papers \cite{hodensities,cilw} contain results that ought to be related to the ones obtained here, although the methods are completely different.
\end{para}

\begin{rem}
	Local contractibility is used in Theorem \ref{mainthm} only to get an isomorphism between sheaf cohomology and singular cohomology. We could have omitted this hypothesis if the right hand sides of the quasi-isomorphisms of the theorem had been interpreted as complexes computing derived functor cohomologies.
\end{rem}

\subsection*{Configuration spaces of points on $i$-acyclic spaces}

\begin{para}
	Our second main theorem is a generalization and re-interpretation of a beautiful result of Arabia \cite{arabia}. Arabia introduced the notion of an \emph{$i$-acyclic space}: a topological space $X$ is $i$-acyclic over a ring $R$ if $H^k_c(X,R) \to H^k(X,R)$ is the zero map for all $k$. This condition is in fact satisfied in many cases of interest: for example, any open subset of Euclidean space is $i$-acyclic, and the product of any space with an $i$-acyclic space is $i$-acyclic. (For example, if $Y$ is arbitrary then $Y \times \mathbf R$ is $i$-acyclic.) More examples are given in Example \ref{acyclic-example}. The remarkable fact about $i$-acyclicity is that it is exactly the right hypothesis to ensure that the compactly supported cohomology of configuration spaces of points on $X$ depends in the simplest possible way on the compactly supported cohomology of $X$ itself. Although the main focus of Arabia's paper is the rational cohomology, the following theorem is proven with arbitrary field coefficients:
\end{para}

\begin{thm}[Arabia]Let $X$ be an $i$-acyclic paracompact locally compact Hausdorff space over a field $k$. Then $H^\bullet_c(F(X,n),k)$ depends only on the graded vector space $H^\bullet_c(X,k)$.\end{thm}

\begin{para}
	On the other hand, a consequence of Theorem \ref{firstmainthm} is that for an arbitrary locally compact Hausdorff space, $H^\bullet_c(F(X,n),\Q)$ depends only on the choice of a cdga model for the compactly supported cochains $C_c^\bullet(X,\Q)$. It would be appealing to try to ``explain'' the rational case of Arabia's result in these terms, instead: that $i$-acyclicity should force the algebra of compactly supported cochains to be homotopically trivial in some strong sense. This is indeed the case:
\end{para}

\begin{thm}[Second main theorem, with $\Q$-coefficients]\label{secondmainthm} Let $X$ be an $i$-acyclic paracompact locally compact Hausdorff space over $\Q$. Then a cdga model for $C_c^\bullet(X,\Q)$ is given by the cohomology $H^\bullet_c(X,\Q)$, with identically zero multiplication and differential. Equivalently, $C_c^\bullet(X,\Q)$ is formal and the cup product in compact support cohomology vanishes.	
\end{thm}

\begin{para}\label{sketch}
	This result re-proves, re-interprets, and generalizes Arabia's theorem. The proof of Theorem \ref{secondmainthm} is simple enough that it makes sense to outline it here in the introduction. It will be enough to construct an $\mathsf A_\infty$-quasi-isomorphism between $H^\bullet_c(X,\Q)$ (with identically zero differential and multiplication) and the dg-algebra $C_c^\bullet(X,\Q)$. This means that we have to find maps
	$$ f_n \colon H_c^\bullet(X,\Q)^{\otimes n} \to C_c^\bullet(X,\Q), \qquad \qquad n \geq 1$$
	of degree $1-n$, such that
	$$ d \circ  f_n = \sum_{i+j=n}(-1)^i f_i \cdot f_j. $$
	Let $f \colon H_c^\bullet(X,\Q) \to C^\bullet_c(X,\Q)$ be a map taking every class to a representing cocycle. Since $X$ is $i$-acyclic, every cocycle in $C_c^\bullet(X,\Q)$ is a coboundary in $C^\bullet(X,\Q)$. This means that there exists $$g \colon H_c^\bullet(X,\Q) \to C^{\bullet-1}(X,\Q),$$ such that $d \circ g = -f$. Now define for all $n \geq 1$,  
	$$ f_n(x) = f(x) \cdot \underbrace{g(x) \cdot \ldots \cdot g(x)}_{(n-1) \text{ times}}. $$
	Since the product of a compactly supported cochain with an arbitrary cochain has compact support, this product is well defined as an element of $C_c^\bullet(X,\Q)$, and one can verify that this defines an $\mathsf A_\infty$-quasi-isomorphism.
\end{para}

\begin{para}By combining Theorems \ref{firstmainthm} and \ref{secondmainthm} we obtain as a corollary that if $X$ is $i$-acyclic, then the spectral sequences of Theorem \ref{firstmainthm} converging to $H^\bullet_c(F(X,U),\Q)$ and $H^\bullet_c(D(X,U),\Q)$ degenerate immediately; there is no differential on any page of the spectral sequence. So, for example, the Totaro spectral sequence degenerates immediately for any $i$-acyclic oriented manifold. We deduce the following: \end{para}

\begin{cor}\label{corollary-intro}Let $X$ be an $i$-acyclic locally contractible paracompact Hausdorff space over $\Q$, and let $U \subset \Pi_n$ be upwards closed. There are isomorphisms
	$$H^k_c(F(X,U),\Q) \cong \bigoplus_{i+j=k}\bigoplus_{T \in J_U} \widetilde H^{i}(\nerveu{J_{U_0}^{\preceq T}},\Q) \otimes H^j_c(X^{\vert T\vert},\Q)$$
	and
	$$H^k_c(D(X,U),\Q) \cong \bigoplus_{i+j=k}\bigoplus_{T \in J_U} \widetilde H^{i}(\nerveu{J_{U}^{\preceq T}},\Q) \otimes H^j_c(X^{\vert T\vert},\Q).$$
	If $G \subset \mathbb S_n$ preserves $U$ then the isomorphisms are $G$-equivariant. 

\end{cor}

\begin{para}If we wish to extend our generalization of Arabia's theorem to arbitrary coefficients, we can no longer work with cdga's. An advantage of the framework set up in this paper is that once the formalism is in place, one can give more or less the same proof as in \S \ref{sketch} to prove a version of Theorem \ref{secondmainthm}  for cohomology with arbitrary coefficients. One might naively hope for the following statement: if $X$ is an $i$-acyclic space over the ring $R$ --- that is, $H^\bullet_c(X,R)\to H^\bullet(X,R)$ is the zero map --- then the tcdga $R\Gamma_c^\otimes(X,R)$ is quasi-isomorphic to its cohomology with identically zero differential and multiplication. However, if the tcdga $R\Gamma_c^\otimes(X,R)$ were homotopically trivial in this strong sense, then not only would the spectral sequence of Theorem \ref{firstmainthm} converging to $H^\bullet_c(F(X,U),R)$ degenerate immediately; we would also have an \emph{equivariant} isomorphism between $H^\bullet_c(F(X,U),R)$ and the direct sum of the terms on the $E_1$-page of the spectral sequence, just as in Corollary \ref{corollary-intro}. Thus we see that the above statement is false already for $X$ the real line and $R=\Z$: for the configuration space of two points we have $H^2_c(F(\R,2),\Z) \cong \Z \oplus \Z$, on which $\mathbb S_2$ acts by switching the two factors. But the $E_1$-page of the spectral sequence has $E_1^{1,1} \cong \Z$ (with the trivial representation of $\mathbb S_2$) and $E_1^{2,0} \cong \Z$ (with the sign representation). Since $\Z \oplus \Z$ is not the direct sum of the trivial and the sign representation, there is no such equivariant isomorphism in this case.
\end{para}

	\begin{para} What one needs to consider instead is the forgetful functor from twisted commutative dg algebras to what is known as \emph{commutative dg shuffle algebras}; a commutative dg shuffle algebra is essentially a twisted commutative algebra in which one has forgotten the actions of the symmetric group while remembering as much as possible of the remaining structure. Our functors $\CF$ and $\CD$ make sense also on the larger category of commutative dg shuffle algebras. Moreover, $i$-acyclicity of a space $X$ \emph{does} imply that $R\Gamma_c^\otimes(X,R)$ is quasi-isomorphic to its cohomology with identically zero multiplication and differential \emph{as a commutative dg shuffle algebra}. 
\end{para}

\begin{thm}[Second main theorem with arbitary coefficients] Let $X$ be a locally compact Hausdorff space, $R$ any ring. Suppose that $X$ is $i$-acyclic over $R$, and that $H^\bullet_c(X,R)$ is a projective $R$-module. Then the tcdga $R\Gamma_c^\otimes(X,R)$ is quasi-isomorphic as a commutative dg shuffle algebra to its cohomology with identically zero multiplication and differential. 
	
\end{thm}

\begin{para}
	Again we obtain as a corollary a theorem describing the compact support cohomology of configuration spaces of points on $i$-acyclic spaces:
\end{para}

\begin{cor}Let $X$ be an $i$-acyclic locally contractible paracompact Hausdorff space over the ring $R$, and let $U \subset \Pi_n$ be upwards closed. There are isomorphisms
	$$H^k_c(F(X,U),R) \cong \bigoplus_{i+j=k}\bigoplus_{T \in J_U} \widetilde H^{i}(\nerveu{J_{U_0}^{\preceq T}};  H^j_c(X^{\vert T\vert},R))$$
	and
	$$H^k_c(D(X,U),R) \cong \bigoplus_{i+j=k}\bigoplus_{T \in J_U} \widetilde H^{i}(\nerveu{J_{U}^{\preceq T}};  H^j_c(X^{\vert T\vert},R)).$$
	If $G \subset \mathbb S_n$ preserves $U$ then the left hand side has a filtration whose associated graded is $G$-equivariantly isomorphic to the right hand side. 
	\end{cor}

\begin{para}In particular, we see under these hypotheses that $H^\bullet_c(F(X,U),R)$ depends only on the graded $R$-module $H^\bullet_c(X,R)$ and on the cohomology of lower intervals in the poset $U$. In any situation where one can compute the cohomology of such lower intervals, the theorem can be used for quite explicit computations. In Section \ref{section:k-equals} of this paper we give very explicit generating series for the rational compact support cohomology of $k$-equals configuration spaces of $i$-acyclic spaces, considered as a sequence of representations of the symmetric groups $\mathbb S_n$. The calculations are expressed in terms of the algebra of symmetric functions; the result is equivalent to computing the ``character polynomials'' of the compactly supported cohomology groups of $k$-equals configuration spaces of $i$-acyclic spaces. The computations of poset cohomology used here are due to Sundaram and Wachs \cite{sundaramwachs}.
	\end{para}

\begin{acknowledgements}This paper has gradually taken shape over a few years, and comments and proddings from several different people have been useful to me, including but not limited to Alexander Berglund, Vladimir Dotsenko, Gijs Heuts, Ben Knudsen, Birgit Richter, Steffen Sagave, Phil Tosteson, and Dylan Wilson. Finally I am grateful to an anonymous referee for their careful reading and useful comments. 
\end{acknowledgements}

\section{Twisted commutative algebras and commutative shuffle algebras}

\subsection*{Twisted commutative algebras}

\begin{para}
	The notion of a twisted commutative algebra goes back a long time, at least to Barratt \cite{barratt-twistedlie} and Joyal \cite{joyalanalyticfunctors}. Twisted commutative algebras have receieved a lot of recent attention because of their role in recent work on stable representation theory by Sam, Snowden and others --- see, for example, \cite{samsnowdenstabilitypatterns} for a possible starting point. In this section we will review the definition and set up our conventions. In particular, we caution the reader of our nonstandard choice that \emph{all our twisted commutative algebras will be non-unital by default.}
\end{para}

\begin{para}\label{definition-tcdga}
	Let $\Fin_+$ be the groupoid of \emph{nonempty} finite sets and bijections. We make $\Fin_+$ into a non-unital symmetric monoidal category using the disjoint union of finite sets. A \emph{twisted commutative dg algebra} (tcdga) is a lax symmetric monoidal functor $A \colon \Fin_+ \to \mathsf{Ch}_R$.  Explicitly, for each nonempty finite set $S$ we are given a cochain complex $A(S)$, we are given functorial chain maps $A(S) \otimes A(T) \to A(S \sqcup T)$ for any pair of nonempty finite sets, and the following two diagrams commute:
		\[\begin{tikzcd}
		A(S) \otimes A(T) \otimes A(U)\arrow[d] \arrow[r] & A(S\sqcup T) \otimes A(U)\arrow[d]\\
		A(S) \otimes A(T \sqcup U) \arrow[r]& A (S \sqcup T \sqcup U),
		\end{tikzcd}\]
		and
	\[\begin{tikzcd}
		A(S) \otimes A(T) \arrow[d, "\cong"] \arrow[r] & A(S\sqcup T) \arrow[d,"\cong"]\\
		A(T) \otimes A(S) \arrow[r]& A (T \sqcup S).
		\end{tikzcd}\]
		A \emph{morphism} of twisted commutative dg algebras is a symmetric monoidal natural transformation. We say that a morphism $A \to A'$ is a \emph{quasi-isomorphism} if $A(S) \to A'(S)$ is a quasi-isomorphism for all $S$. 
\end{para}

\begin{para}
	If $A$ is a tcdga, then so is its cohomology $H^\bullet(A)$, where we set $H^\bullet(A)(S) = H^\bullet(A(S))$.  
\end{para}

\begin{para}\label{explicit-tca}
	A skeleton of the category $\mathsf{Fin}_+$ is given by the disjoint union of all symmetric groups $\mathbb S_n$, $n \geq 1$. It follows that a tcdga can also be described as a sequence of $\mathbb S_n$-representations $A(n)$ in the category of dg $R$-modules, together with $\mathbb S_n \times \mathbb S_m$-equivariant multiplication maps
	$$ A(n) \otimes A(m) \to A(n+m)$$
	which are associative, and for which the following diagram commutes:
	\[\begin{tikzcd}
	A(n) \otimes A(m) \arrow[d] \arrow[r] & A(n+m) \arrow[d]\\
	A(m) \otimes A(n) \arrow[r]& A (m+n).
	\end{tikzcd}
	\]The left vertical arrow is the map switching the two tensor factors (and taking into account the Koszul sign rule), and the right vertical arrow is given by acting with the ``box permutation'' in $\mathbb  S_{n+m}$ that moves the first $n$ elements past the last $m$ elements. We refer to $A(n)$ as the \emph{arity $n$ component} of the tcdga $A$. We will switch freely between both definitions of a tcdga. 
\end{para}

\begin{para}\label{constant-tcdga}
	Let $\Omega$ be a not necessarily unital commutative dg algebra over $R$. We can associate to $\Omega$ a ``constant'' tcdga $\underline \Omega$ by the rule that $\underline \Omega(S) = \Omega$ for every nonempty finite set $S$, any bijection $S \to T$ is mapped to the identity map $\Omega \to \Omega$, and the maps $\underline \Omega(S) \otimes \underline \Omega(T) \to \underline \Omega(S \sqcup T)$ are given by the multiplication in $\Omega$. In this way we can think of the category of tcdga's as an enlargement of the category of cdga's. Whenever we say that we consider a cdga $\Omega$ as a tcdga, this is the construction we have in mind. 
\end{para}

\begin{para}
	Later, we will also briefly consider arbitrary lax monoidal functors $\Fin_+ \to \mathsf{Ch}_R$, not just lax symmetric monoidal functors. A lax monoidal functor $A \colon \Fin_+ \to \mathsf{Ch}_R$ is called a \emph{twisted associative dg algebra}.
\end{para}

\begin{para}\label{modules-defn}
	A definition equivalent to \S \ref{definition-tcdga} is that a tcdga is a reduced left module over the  operad $\mathsf{Com}$ of non-unital commutative algebras; similarly, a twisted associative dg algebra is a reduced left module over the operad $\mathsf{Ass}$ of non-unital associative algebras. More generally, if $\EuScript P$ is a dg operad, then left modules over $\EuScript P$ are sometimes called \emph{twisted $\EuScript P$-algebras}. The terminology goes back to Barratt \cite{barratt-twistedlie}. 
\end{para}


\subsection*{Shuffle algebras}

\begin{para}In Section \ref{section:arabia} we will study the cohomology of configuration spaces of points on $i$-acyclic spaces, with coefficients in an arbitrary ground ring $R$. There will be technical complications arising from the fact that the trivial representation of $\mathbb S_n$ is not a projective $R[\mathbb S_n]$-module. For this reason we will need to work also with a notion of a twisted commutative algebra in which one has ``forgotten'' about the actions of the symmetric group --- more precisely, we will need to work with \emph{commutative shuffle dg algebras}. Commutative shuffle algebras are significantly less studied than twisted commutative algebras: shuffle algebras (not necessarily commutative) were introduced by Ronco \cite[Section 2]{ronco}, and to my knowledge, the only explicit mention of their evident commutative analogue in the literature is in \cite[\S 4.6.3]{bremnerdotsenko}. (A word of caution is that there exists several completely unrelated notions of ``shuffle algebra'' in the literature.) The reader who is only interested in working rationally may skip this subsection.
\end{para}

\begin{para}\label{monoidalstructure}Consider the functor category $\mathsf{Ch}_R^{\mathsf{Fin}_+}$, equipped with the monoidal structure given by Day convolution:
	$$ (A \otimes B)(S) = \bigoplus_{S = T \sqcup T'} A(T) \otimes B(T').$$
	It is a general fact about Day convolution that commutative monoids in this monoidal category can be canonically identified with lax symmetric monoidal functors $\Fin_+ \to \mathsf{Ch}_R$, which gives yet another reformulation of the definition of a tcdga. Similarly, not necessarily commutative monoids in this category may be identified with twisted associative dg algebras. 
\end{para}

\begin{para}\label{diagram}
	Let $\mathsf{Ord}_+$ denote the category of finite nonempty totally ordered sets and order-preserving bijections. The functor category $\mathsf{Ch}_R^{\mathsf{Ord}_+}$ can be given a symmetric monoidal structure analogous to the one on $\mathsf{Ch}_R^{\mathsf{Fin}_+}$: viz., 
	$$ (A \otimes B)(S) = \bigoplus_{S = T \sqcup T'} A(T) \otimes B(T').$$
	In this formula, $S = T \sqcup T'$ should be read as saying that $S$ is the union of two nonempty disjoint subsets $T$, $T'$, and that the total orders on $T$ and $T'$ are the ones inherited from $S$. But we do \emph{not} insist that every element in $T$ is smaller than every element of $T'$; that is, $S$ is not in any sense the coproduct of $T$ and $T'$ in the category $\mathsf{Ord}_+$. 
	\end{para}
		
	\begin{defn}\label{defn-shuffle}
	Commutative monoids in the monoidal category $\mathsf{Ch}_R^{\mathsf{Ord}_+}$ are called \emph{commutative shuffle dg algebras}. Not necessarily commutative monoids are called \emph{shuffle dg algebras}.	
	\end{defn}

	\begin{rem}
		One motivation for introducing the above monoidal product is that it makes the forgetful functor $\mathsf{Ch}_R^{\mathsf{Fin}_+} \to \mathsf{Ch}_R^{\mathsf{Ord}_+}$ strong symmetric monoidal. This implies that the forgetful functor from twisted commutative algebras to commutative shuffle algebras preserves certain algebraic properties that the forgetful functor to $\mathbf N$-graded associative algebras does not; 	for example, a presentation of a twisted commutative algebra by generators and relations will also be a presentation of its underlying commutative shuffle algebra. If we had considered the tensor product on $\mathsf{Ch}_R^{\mathsf{Ord}_+}$ given by Day convolution, then the forgetful functor would only be lax symmetric monoidal. 
	\end{rem}

	\begin{para}
		One can make the notion of a (commutative) shuffle dg algebra more explicit, like we did for tcdga's in \S \ref{explicit-tca}. A shuffle dg algebra is a collection of dg $R$-modules $A(n)$, $n \geq 1$, and for every shuffle permutation $\pi \in \mathrm{Sh}(n,m) \subset \mathbb S_{n+m}$ a composition map
		$$ \circ_\pi \colon A(n) \otimes A(m) \to A(n+m),$$
		satisfying the following associativity law: any shuffle permutation $\pi \in \mathrm{Sh}(n,m,l)$ can be uniquely written as $\pi = (\mathbf 1_n \times \sigma) \circ \rho$ where $\sigma \in \mathrm{Sh}(m,l)$ and $\rho \in \mathrm{Sh}(n,m+l)$, and also as $\pi = (\tau \times \mathbf 1_l) \circ \lambda$ with $\tau \in \mathrm{Sh}(n,m)$ and $\lambda \in \mathrm{Sh}(n+m,l)$, and there is an equality
		$$ x \circ_\rho (y \circ_\sigma z) = (x \circ_\tau y) \circ_\lambda z$$
		for any $x \in A(n)$, $y \in A(m)$, $z \in A(l)$. A shuffle dg algebra is said to be commutative if it has the following additional property: if we denote the evident bijection $\mathrm{Sh}(n,m) \to \mathrm{Sh}(m,n)$ by $\pi \mapsto \pi'$,  then we should have $x \circ_\pi y = y \circ_{\pi'} x$ for any $\pi \in \mathrm{Sh}(n,m)$, $x \in A(n)$ and $y \in A(m)$. 
	\end{para}

	\begin{para}
		There is an evident diagram of forgetful functors:
	\[\begin{tikzcd}
	(\text{twisted commutative dg algebras}) \arrow[d] \arrow[r] & (\text{twisted associative dg algebras}) \arrow[d]\\
	(\text{commutative shuffle dg algebras}) \arrow[r]& (\text{shuffle dg algebras}).
	\end{tikzcd}\]
	
\end{para}

\begin{rem}As noted in \S \ref{modules-defn}, twisted commutative algebras are the same thing as left modules over the commutative operad. Although we will not use it in the sequel, let us mention that there is an analogous description of commutative shuffle algebras, using \emph{shuffle operads.} A shuffle operad is essentially a symmetric operad in which one has forgetten the actions of the symmetric group in a minimally destructive manner; they interpolate between symmetric and nonsymmetric operads. Shuffle operads were introduced by Dotsenko and Khoroshkin \cite{dotsenkokhoroshkin} for the purposes of developing a Gr\"obner theory for operads. Let $(-)^f$ denote the forgetful functor from symmetric to shuffle operads. Then commutative shuffle dg algebras are the same thing as left modules over the shuffle operad $\mathsf{Com}^f$, and shuffle dg algebras are the same thing as left modules over $\mathsf{Ass}^f$.
\end{rem}

\section{Twisted commutative cochains}
\label{section:commutativecochains}
\subsection*{The functors $R\Gamma^\otimes$ and $R\Gamma_c^{\otimes}$}

\begin{para}
	Before proceeding further it is perhaps useful to quickly recall some of the basics of sheaf cohomology, since this is not always part of the ``standard toolbox'' for a working algebraic topologist. Fix a topological space $X$. We have left exact functors $\Gamma$ and $\Gamma_c$ from sheaves of $R$-modules on $X$ to $R$-modules, given by taking global sections and global sections with compact support, respectively. Cohomology, and cohomology with compact support, of a sheaf of $R$-modules is defined as the derived functors of $\Gamma$ and $\Gamma_c$. These derived functors can be defined using injective resolutions, but in practice they are usually computed by some other acyclic resolution. A standard choice is the Godement resolution, which has several pleasant formal properties: it is functorial, and the monoidal structure of the Godement resolution allows for the construction of functorial cup-products on the chain level. The sheaves in the Godement resolution are flabby, meaning that all restriction maps are surjective.  This implies that they are acyclic for both $\Gamma$ and $\Gamma_c$. We define $R\Gamma(X,-)$ and $R\Gamma_c(X,-)$ to be the global sections (resp.\ global sections with compact support) of the Godement resolution of a sheaf or complex of sheaves. It is somewhat more natural to consider all of the above as taking place in the derived category of sheaves of $R$-modules, i.e.\ all operations are taken on complexes of sheaves rather than on individual sheaves, and quasi-isomorphisms are formally inverted. We will also need to consider the derived tensor product of sheaves of $R$-modules, which is instead computed by a flat resolution; a sheaf is said to be flat if all its stalks are flat $R$-modules.
\end{para}

\begin{para}
	Let $X$ be a topological space, and $\EuScript F$ a complex of sheaves of $R$-modules on $X$. We define two twisted commutative dg-algebras $\Gamma^\otimes(X,\EuScript F)$ and $\Gamma^\otimes_c(X,\EuScript F)$ associated to this data: we let 
	$$ \Gamma^\otimes(X,\EuScript F)(S) = \Gamma(X,\EuScript F^{\otimes S}) \qquad \text{and} \qquad \Gamma^\otimes_c(X,\EuScript F)(S) = \Gamma_c(X,\EuScript F^{\otimes S}).$$
	The multiplication maps are the obvious ones: a section of $\EuScript F^{\otimes S}$ and a section of $\EuScript F^{\otimes T}$ can be multiplied together, to produce a section of $\EuScript F^{\otimes S \sqcup T}$. 
	
	Alternatively, we may think of this construction as first taking a complex of sheaves $\EuScript F$, and constructing a sheaf of tcdga's on $X$, given by $S \mapsto \EuScript F^{\otimes S}$ (essentially the tensor algebra on $\EuScript F$). Moreover, as the functors $\Gamma(X,-)$ and $\Gamma_c(X,-)$ are lax symmetric monoidal\footnote{Note, though, that $\Gamma_c$ is only a non-unital monoidal functor.}, the global sections (with or without compact support) of a sheaf of tcdga's is itself a tcdga. 
	\end{para}

	\begin{para}\label{commutative-cochains}
		Even though $\Gamma$ and $\Gamma_c$ are lax symmetric monoidal functors, their derived versions $R\Gamma$ and $R\Gamma_c$ are not --- if they were, $R\Gamma(X,\mathbf Z)$ would be a commutative cochain model for the space $X$, and it is well known that one cannot in general construct strictly commutative cochains unless one works over a field of characteristic zero. Nevertheless we can define derived versions of the functors $\Gamma^\otimes$ and $\Gamma_c^{\otimes}$, producing a strictly commutative tcdga. This means that the ``commutative cochain problem'' can always be solved in the larger category of \emph{twisted} dg algebras. 
	\end{para}

			\begin{lem}\label{flabbyflat}
				Let $X$ be a topological space, $R$ a ring of finite global dimension. Let $\EuScript F$ be a bounded below complex of sheaves of $R$-modules on $X$. Then $\EuScript F$ is functorially quasi-isomorphic to a bounded below complex of flat and flabby sheaves. 
			\end{lem}
			
			\begin{proof}
				First of all, the fact that $R$ is of global dimension $d$ implies that every $R$-module $M$ admits a functorial free resolution of length $\leq d$: indeed, one can form a functorial free resolution in the usual way and then truncate, and the dimension hypothesis assures us that we still have a free resolution. It follows from this that every bounded below complex of $R$-modules has a functorial free resolution which is also bounded below. Globally on $X$ this shows that any bounded below complex of presheaves of $R$-modules on $X$ has a functorial bounded below resolution by representable presheaves. Representable presheaves are in particular flat, so $\EuScript F$ can be resolved (as a presheaf) in a functorial manner by a bounded below complex of flat presheaves. But sheafification is an exact functor and preserves flatness, since flatness is a condition on stalks. So we obtain such a resolution also in the category of sheaves. Now we form the Godement resolution of the resulting bounded below complex, and use that the Godement resolution of a flat sheaf is again flat \cite[VI.1, Proof of Proposition 1.3]{iversensheaves}.
			\end{proof}
		\begin{defn}\label{definition-rgamma}
			Let $X$ be a topological space, $\EuScript F$ a bounded below complex of sheaves of $R$-modules on $X$, $\mathrm{gldim}(R)<\infty$. Let $\EuScript L$ be the flat and flabby resolution of $\EuScript F$ constructed in Lemma \ref{flabbyflat}. Define
			$$ R\Gamma^\otimes(X,\EuScript F) = \Gamma^{\otimes}(X,\EuScript L) \qquad \text{and} \qquad R\Gamma^\otimes_c(X,\EuScript F) = \Gamma_c^\otimes(X,\EuScript L).$$
			This definition is not the ``right'' one, however, unless we impose some point-set hypotheses on $X$:
		\end{defn}	
		
		\begin{prop}\label{rightcohomology}
			Let $X$, $\EuScript F$ and $R$ be as in Definition \ref{definition-rgamma}. If $X$ is paracompact and Hausdorff, then $$ H^\bullet(R\Gamma^\otimes(X,\EuScript F))(S) \cong H^\bullet(X,\underbrace{\EuScript F \otimes^L \ldots \otimes^L \EuScript F}_{S \text{ factors}}).$$ If $X$ is moreover locally compact, then  \[H^\bullet( R\Gamma^\otimes_c(X,\EuScript F))(S) \cong H^\bullet_c(X,\underbrace{\EuScript F \otimes^L \ldots \otimes^L \EuScript F}_{S \text{ factors}}).\]
		\end{prop}
			
	\begin{proof}Let $\EuScript L$ be the resolution of $\EuScript F$ of Lemma \ref{flabbyflat}. Then $\EuScript L^{\otimes n} \simeq \EuScript F \otimes^L \ldots \otimes^L \EuScript F$ ($n$  factors) for all $n$, so to deduce the result we need to know that $\EuScript L^{\otimes n}$ is a $\Gamma$-acyclic (resp.\ $\Gamma_c$-acyclic) resolution. This is not automatic: although $\EuScript L $ was a complex of flabby sheaves, the tensor product of two flabby sheaves is not necessarily flabby. However, flabby sheaves are soft, any tensor product with a flat and soft sheaf is again soft, and soft sheaves are $\Gamma$-acyclic. Similarly, if $X$ is locally compact then flabby sheaves are $c$-soft, any tensor product with a flat and $c$-soft sheaf is $c$-soft, and $c$-soft sheaves are $\Gamma_c$-acyclic. 
		\end{proof}
	\subsection*{Commutative cochains over $\Q$}
	
	\begin{para}As remarked above, the functors $R\Gamma$ and $R\Gamma_c$ are not lax symmetric monoidal: if they were, one could construct strictly commutative cochains on any space. On the other hand, strictly commutative cochains do exist over $\Q$, as constructed by Quillen \cite{quillenrationalhomotopytheory} and Sullivan \cite{sullivaninfinitesimal}. Hence one might expect that when working with sheaves of $\Q$-vector spaces, the functors $R\Gamma$ and $R\Gamma_c$ can be replaced by commutative versions. This is indeed the case, using a construction due to Navarro Aznar \cite[Section 4]{navarroaznarhodgedeligne}. \end{para}
	
	\begin{thm}[Navarro Aznar]\label{navarro-aznar-theorem}
		Let $X$ be a topological space, $C_\Q^+(X)$ the category of bounded below complexes of sheaves of $\Q$-vector spaces on $X$. There are exact lax symmetric monoidal functors 
		$$ R\Gamma^{TW}, R\Gamma_c^{TW} \colon C_\Q^+(X) \to \mathsf{Ch}^+_\Q$$
		equipped with monoidal quasi-isomorphisms of functors $R\Gamma \implies R\Gamma^{TW}$, $R\Gamma_c \implies R\Gamma^{TW}_c$. In particular, if $A$ is a sheaf of commutative dg  $\Q$-algebras on $X$, then $R\Gamma^{TW}(X,A)$ and $R\Gamma_c^{TW}(X,A)$ are themselves commutative dg algebras. 
	\end{thm}
	
	\begin{proof}[Sketch of proof]Recall how to define the Godement resolution: there is a monad $\mathsf T$ on the category of sheaves on $X$ given by $p_\ast p^\ast$, where $p \colon X^\delta \to X$ is the projection from $X$ equipped with the discrete topology. With this monad, one attaches to a any sheaf $\EuScript F$ on $X$ a cosimplicial sheaf $\mathsf T(\EuScript F)$ \cite[8.6.15]{weibel}, whose totalization is the Godement resolution.  
		
	Thus the usual functor $R\Gamma$ is the composition of three functors: the cosimplicial Godement construction, the totalization functor on cosimplicial objects, and the functor $\Gamma$. Of these, the first and the last are already lax symmetric monoidal. This reduces the problem to finding a commutative version of the totalization functor on cosimplicial objects, which works over $\Q$. This is indeed what Navarro Aznar does \cite[\S\S 2--3]{navarroaznarhodgedeligne}, by constructing the so-called ``Thom--Whitney totalization'',  using Sullivan's polynomial de Rham forms. To be more precise, the usual totalization functor assigns to a cosimplicial object $A \colon \Delta \to \mathsf{Ch}_R$ the end of the bifunctor $\Delta^{\mathrm{op}} \times \Delta \to \mathsf{Ch}_R$ given by $(\alpha,\beta) \mapsto C^\bullet(\Delta^\alpha,R) \otimes A^\beta$, where $C^\bullet(\Delta^\alpha,R)$ denotes the standard cochains on the $\alpha$-simplex. The reason that totalization is monoidal but not symmetric monoidal is that $C^\bullet(\Delta^\ast,R)$ is a simplical dg algebra, but it is not commutative. Over $\Q$ we may replace  $C^\bullet(\Delta^\ast,\Q)$ with the simplical commutative dg algebra $\Omega(\Delta^\ast)$ given by Sullivan's polynomial de Rham forms \cite{sullivaninfinitesimal}, which gives a lax symmetric monoidal totalization functor equipped with a monoidal quasi-isomorphism with the usual totalization.

	Hence if $\EuScript F$ is a sheaf of commutative dg $\Q$-algebras on $X$, its Godement resolution,  defined using the Thom--Whitney totalization described above, is again a sheaf of commutative dg $\Q$-algebras, and a flabby resolution of the original complex of sheaves. 
	
	One argues in the same way for the functor $R\Gamma_c$. 
	\end{proof}

\begin{thm}
	Let $X$ be a paracompact Hausdorff space. The tcdga $R\Gamma^\otimes(X,\Q)$ is quasi-isomorphic to the cdga $R\Gamma^{TW}(X,\Q)$, considered as a ``constant'' tcdga as in \S \ref{constant-tcdga}. If $X$ is moreover locally compact then $R\Gamma^\otimes_c(X,\Q)$ is quasi-isomorphic to the cdga $R\Gamma^{TW}_c(X,\Q)$.
\end{thm}

\begin{proof}
	Let $\mathrm{Gode}^{TW}(\Q)$ be the Godement resolution of the sheaf $\Q$ on $X$, defined using Navarro Aznar's Thom--Whitney totalization, as in the proof of Theorem \ref{navarro-aznar-theorem}. Note that the lax symmetric monoidal structure of the functor $\mathrm{Gode}^{TW}$ defines for any $n$ an $\mathbb S_n$-equivariant quasi-isomorphism $$\mathrm{Gode}^{TW}(\Q)^{\otimes n} \stackrel{\sim}\to \mathrm{Gode}^{TW}(\Q).$$ Now we may compute the tcdga $R\Gamma^\otimes (X,\Q)$ using the resolution $\mathrm{Gode}^{TW}(\Q)$, so that for $S \in \mathrm{ob}\,\, \mathsf{Fin}_+$ we have 
	$$ R\Gamma^\otimes(X,\Q)(S) \simeq \Gamma(X,\mathrm{Gode}^{TW}(\Q)^{\otimes S}) \stackrel \sim \to \Gamma(X,\mathrm{Gode}^{TW}(\Q)) \simeq R\Gamma^{TW}(X,\Q) $$
	where the quasi-isomorphism above is induced by the symmetric monoidal structure. 
\end{proof}

\begin{lem}\label{excision}
	Let $X$ be a compact Hausdorff space, $A \subset X$ a closed subset. There is a natural quasi-isomorphism $R\Gamma_c^{TW}(X \setminus A,\Q) \stackrel \sim \to  \mathrm{cone}(R\Gamma^{TW}(X,\Q) \to R\Gamma^{TW}(A,\Q)).$ 
\end{lem}

\begin{proof}
	Let $j \colon (X \setminus A) \to X$ and $i \colon A \to X$ be the inclusions. We get a short exact sequence 
	$$ 0 \to j_! \Q \to \Q \to i_\ast \Q \to 0$$
	of sheaves on $X$. Applying the functor $R\Gamma^{TW}(X,-)$ 
	gives the result, since $R\Gamma_c^{TW}(X \setminus A,\Q) = R\Gamma^{TW}(X,j_!\Q)$ (here we use $X$ compact) and $R\Gamma^{TW}(A,\Q) = R\Gamma^{TW}(X,i_\ast\Q)$.
\end{proof}

\begin{para}
	We claim that the functor $R\Gamma^{TW}(-,\Q)$ is a cochain theory on the category of CW complexes in the sense of Mandell \cite{mandellcochain}. Indeed, the functor $R\Gamma$ is homotopy invariant \cite[IV.1]{iversensheaves}, hence a weak homotopy invariant of CW complexes by Whitehead's theorem. The wedge axiom is clear. It is well known that in the presence of the wedge axiom/direct limit axiom, it is enough to verify exactness/excision for \emph{finite} CW complexes. But if $(X,A)$ is a pair of finite CW complexes, then exactness and excision is immediate from the preceding lemma.
\end{para}

\begin{thm}\label{APL}Let $X$ be a CW complex. The cdga $R\Gamma^{TW}(X,\Q)$ is quasi-isomorphic to Sullivan's cdga $A_{PL}(X)$. If $X$ is the complement of a subcomplex in a finite CW complex then $R\Gamma_c^{TW}(X,\Q)$ is quasi-isomorphic to the augmentation ideal in the algebra $A_{PL}(X^\ast)$, where $X^\ast$ denotes the one-point compactification.
\end{thm}

\begin{proof}Both functors $A_{PL}$ and $R\Gamma^{TW}$ are cochain theories. Then the first claim follows from Mandell's uniqueness theorem \cite[Corollary, p. 550]{mandellcochain}, which says in particular that there is a unique cochain theory taking values in commutative dg $\Q$-algebras, up to quasi-isomorphism. The second claim follows from the first and Lemma \ref{excision}.
\end{proof}

%

\begin{cor}\label{constant-APL}Let $X$ be a CW complex. The tcdga $R\Gamma^{\otimes}(X,\Q)$ is quasi-isomorphic to the constant tcdga associated to a cdga model for the cochains $C^\bullet(X,\Q)$. If $X$ is the complement of a subcomplex in a finite CW complex then $R\Gamma^{\otimes}_c(X,\Q)$ is quasi-isomorphic to a cdga model for the compactly supported cochains $C^\bullet_c(X,\Q)$.
\end{cor}

\begin{proof}
	Indeed, $A_{PL}(X)$ is a cdga model for $C^\bullet(X,\Q)$, and the augmentation ideal in $A_{PL}(X^\ast)$ is a cdga model for $C^\bullet_c(X,\Q)$.
\end{proof}

\subsection*{The relationship with $\mathbb E_\infty$-algebras} 

\begin{para}There is a close relationship between twisted commutative dg algebras and $\mathbb E_\infty$-algebras. We briefly state the main results here; these are principally due to Sagave--Schlichtkrull \cite{sagaveschlichtkrull} (who only treated $\mathbb E_\infty$-algebras in the category of spaces, see Richter--Shipley \cite{richtershipley} for the analogous results for cochain complexes) and Pavlov--Scholbach \cite{pavlovscholbach}. The results of this section will not be used in the sequel and are included only to put our constructions into a broader context. We will first state the results in the \emph{non-unital setting}, since the statements are cleaner in this case, and the algebras of interest to us are non-unital. 
\end{para}

\begin{defn}
	Let $\I_+$ be the category of nonempty finite sets and injections. The category $\I_+$ is a (non-unital) symmetric monoidal category under disjoint union. A \emph{dg $\I_+$-algebra} is a lax symmetric monoidal functor $\I_+ \to \mathsf{Ch}_R$. 
\end{defn}

\begin{para}
	Every dg $\I_+$-algebra has an underlying tcdga, given by restriction along the functor $\mathsf{Fin}_+ \to \I_+$. The forgetful functor from dg $\I_+$-algebras to twisted commutative dg algebras has a left adjoint given by the left Kan extension. 
\end{para}

\begin{para}
	The following theorem can be deduced as a very special case of general results of Pavlov--Scholbach \cite{pavlovscholbach}, although they do not state it in the precise form we need it. Indeed their goal is to construct a model structure on operadic algebras in abstract symmetric spectra with respect to the \emph{stable} model structure on spectra, whereas Theorem \ref{ps} concerns the \emph{unstable} model structure, from which the stable model structure is obtained by localization. But all the properties that they use of the unstable model structure are first verified for the stable model structure. Moreover, they consider throughout \emph{unital} objects, which in this context means that they consider the categories $\mathsf{Ch}_R^{\mathsf{Fin}}$ and $\mathsf{Ch}_R^{\mathsf{I}}$, where the empty set is included in the domain category. This forces them to work with a ``{positive}'' model structure, in which a map $A \to B$ is a weak equivalence (fibration) if $A(S) \to B(S)$ is a weak equivalence (fibration) for \emph{nonempty} $S$. The categories $\mathsf{Ch}_R^{\mathsf{Fin}_+}$ and  $\mathsf{Ch}_R^{\mathsf{I}_+}$ are Quillen equivalent as symmetric monoidal non-unital model categories to $\mathsf{Ch}_R^{\mathsf{Fin}}$ and  $\mathsf{Ch}_R^{\mathsf{I}}$, respectively, with their positive model structures, and this implies with only a little bit of work also a Quillen equivalence between categories of $P$-algebras in the respective categories. 
\end{para}

\begin{thm}
	\label{ps}
	The category of twisted commutative dg algebras over a ring $R$ and the category of dg $\I_+$-algebras both admit model structures in which the weak equivalences are the quasi-isomorphisms and the fibrations are the degreewise surjections. In particular, the forgetful functor from dg $\I_+$-algebras to tcdga's is a right Quillen functor. More generally, if $P$ is any operad with $P(0)=0$, then the categories of $P$-algebras in $\mathsf{Ch}_R^{\mathsf{Fin}_+}$ and in $\mathsf{Ch}_R^{\mathsf{I}_+}$ both admit such model structures.
\end{thm}

\begin{rem}Since both $\mathsf{Ch}_R^{\mathsf{Fin}_+}$ and $\mathsf{Ch}_R^{\mathsf{I}_+}$ are \emph{non-unital} symmetric monoidal categories, the notion of a $P$-algebra in these categories does not even make sense unless $P(0)=0$. Indeed, a $P$-algebra is defined by structure maps $P(n) \otimes A^{\otimes n} \to A$, and in a non-unital monoidal category the tensor power $A^{\otimes 0}$ is undefined. \end{rem}


\begin{proof}[Sketch of proof of Theorem \ref{ps}]Let $\EuScript C$ denote either of the two categories $\mathsf{Ch}_R^{\mathsf{Fin}_+}$ or $\mathsf{Ch}_R^{\mathsf{I}_+}$. The strategy is to lift the projective model structure on $\EuScript C$ along the forgetful-free adjunction $$U \colon \mathrm{Alg}(P) \leftrightarrows \EuScript C \colon F_P.$$ The nontrivial property that needs to be checked in order for the model structure to lift is that if $M \to N$ is a trivial cofibration in $\EuScript C$, $A$ is a $P$-algebra, and $F_P(M) \to A$ is any $P$-algebra morphism  then the pushout
	$$ A \to A \amalg_{F_P(M)} F_P(N)$$
is a trivial cofibration. We may in fact take $M \to N$ to be a generating trivial cofibration, so $0=M=F_P(M)$ and we are computing a coproduct. Now for any $P$-algebra $A$ in $\EuScript C$ there is an operad $O_A$ in $\EuScript C$ such that $O_A$-algebras are equivalent to $P$-algebras with a morphism from $A$ \cite[Lemma 1.18]{getzlerjones}. The arity $0$ component of $O_A$ is isomorphic to $A$, since it is the initial object in the category of $O_A$-algebras. The free $O_A$-algebra on an object $N$ of $\EuScript C$ is the pushout $A \amalg F_P(N)$. Now
$$ F_{O_A}(N) = A \oplus \bigoplus_{n \geq 1} O_A(n) \otimes_{\mathbb S_n} N^{\otimes n}$$
and we claim that $ O_A(n) \otimes_{\mathbb S_n} N^{\otimes n}$ is acyclic for all $n \geq 1$. Indeed, 
$$ N^{\otimes n}(k) =\!\! \bigoplus_{\{1,\ldots,k\} = S_1 \sqcup \ldots \sqcup S_n}\!\! N(S_1) \otimes \ldots \otimes N(S_n) = R[\mathbb S_n] \otimes \!\!\bigoplus_{\substack{\{1,\ldots,k\} = S_1 \sqcup \ldots \sqcup S_n\\ \min(S_1) < \ldots < \min(S_n)}}\!\! N(S_1) \otimes \ldots \otimes N(S_n)$$ 
and each $N(S_i)$ is an acyclic complex of free $R$-modules. Note that in the last step we used crucially that $N(0)=0$, to ensure that none of the subsets $S_i$ were empty. 
\end{proof}

\begin{para}Theorem \ref{ps} is perhaps surprising at first, since $R$ is a completely arbitrary ring. Hinich \cite{hinichmodelstructure} and Harper \cite{harpermodelstructure} proved analogous results for $P$-algebras in $\mathsf{Ch}_R$ and in $\mathsf{Ch}_R^{\mathsf{Fin}}$ (i.e. left $P$-modules) only when $R$ contains $\Q$, and without this assumption there is in general no such model structure \cite[Example 3.7]{goerssschemmerhorn}. Theorem \ref{ps} shows that the arity $0$ component is in fact the only obstruction to transferring the model structure: once we impose the condition that $M(0)=0$, everything works without a hitch over an arbitrary base ring. Thus twisted commutative dg algebras are ``better behaved'' than ordinary commutative algebras in positive characteristic, as evidenced in particular by the fact that the commutative cochain problem can always be solved in the category of tcdga's, as we explained in \S \ref{commutative-cochains}. There are in fact several earlier results in the literature of the following flavor: for a dg operad $P$, left $P$-modules $M$ with $M(0)=0$ exhibit ``homotopically correct'' behavior in arbitrary characteristic, even in situations where $P$-algebras do not. We mention three examples:\begin{enumerate}
		\item Stover \cite{stover} proved that reduced left $\mathsf{Lie}$ modules satisfy the Milnor--Moore theorem over an arbitrary base ring; the classical Milnor--Moore theorem for Lie algebras requires working over a $\Q$-algebra.
		\item Richter \cite{richtershuffle} proved that if $A$ is a reduced left $\mathsf{Com}$-module (i.e.\ a non-unital tcdga), then the Harrison homology, Andr\'e--Quillen homology and Gamma-homology of $A$ all coincide, and give a summand of Hochschild homology. For usual commutative algebras this is in general only true over $\Q$. 
		\item Fresse \cite{fressesimplicial} introduced the notion of a {divided power structure} on a $P$-algebra. The definition admits an evident generalization to a divided power structure on a left $P$-module $M$. If $M(0)=0$, then it follows from \cite[9.10]{stover}  that $M$ is always \emph{canonically} equipped with a divided power structure.  
	\end{enumerate} \end{para} 

\begin{para}We are now ready to state the relationship between $\I_+$-algebras and $\mathbb E_\infty$-algebras, which can be summed up in the following two theorems, which are simple modifications of the statements of \cite[Theorem 2.13]{richtersagave} and \cite[Theorem 4.10]{richtersagave}:
\end{para}

\begin{thm}\label{non-unitalthm}
	Let $A$ be a dg $\I_+$-algebra. Then the Bousfield-Kan homotopy colimit $\hocolim_{\I_+} A$ admits a natural action by the Barratt--Eccles operad, making it into a non-unital $\mathbb E_\infty$-algebra. 
\end{thm}

\begin{thm}\label{bousfield}The category of non-unital $\mathbb E_\infty$-algebras can be identified, via a zig-zag of Quillen equivalences, with a Bousfield localization of the category of dg $\I_+$-algebras. The functor $\hocolim_{\I_+}$ from dg $\I_+$-algebras to non-unital $\mathbb E_\infty$-algebras models the composite of the derived functors arising from this zig-zag. The fibrant objects in the Bousfield localization  are the $\I_+$-algebras for which $A(S) \to A(T)$ is a quasi-isomorphism for every injection $S \hookrightarrow T$. 
\end{thm}

\begin{para}
	Concretely, Theorem \ref{bousfield} says that we can think of the category of non-unital $\mathbb E_\infty$-algebras as a full subcategory of dg $\I_+$-algebras. In particular, there is then also a natural forgetful functor from non-unital $\mathbb E_\infty$-algebras to twisted commutative dg algebras, by Theorem \ref{ps}. If $\Omega$ is a non-unital $\mathbb E_\infty$-algebra, then by Theorem \ref{bousfield} we may associate to $\Omega$ a fibrant dg $\I_+$-algebra $\underline \Omega$, well defined up to quasi-isomorphism. Since all the maps $\underline \Omega(S) \to \underline \Omega(T)$ are quasi-isomorphisms, each component $\underline \Omega(S)$ is in fact quasi-isomorphic to the homotopy colimit $\Omega$, and the multiplication maps $\underline\Omega(S) \otimes \underline{\Omega}(T) \to \underline{\Omega}(S \sqcup T)$ are given up to homotopy by the multiplication in $\Omega$. Thus we may think of this as a homotopy coherent version of the construction of \S \ref{constant-tcdga}. Composing with the forgetful functor from dg $\I_+$-algebras to twisted commutative algebras, we may in particular think of a non-unital $\mathbb E_\infty$-algebra as defining a ``homotopically constant'' tcdga. 
\end{para}

\begin{para}\label{richtersagavesummary}Let us explain how Theorems \ref{non-unitalthm} and \ref{bousfield}, which concern non-unital $\mathbb E_\infty$-algebras, follow from analogous theorems for unital $\mathbb E_\infty$-algebras. We briefly recall these results here; for a more complete account with detailed references see \cite[Section 4]{richtersagave}. Let $\I$ denote the unital symmetric monoidal category of all finite sets and injections. The category of functors $\mathsf I \to \mathsf{Ch}_R$ can be given a \emph{positive} model structure in which $F \to G$ is a fibration (weak equivalence) if $F(S) \to G(S)$ is so for all non-empty sets $S$; no assumption is made on $F(\varnothing) \to G(\varnothing)$. We say that $F \to G$ is an \emph{$\I$-equivalence} if $\hocolim_\I F \to \hocolim_\I G$ is a weak equivalence. The positive model structure admits a Bousfield localization with weak equivalences the $\mathsf I$-equivalences, which we call the \emph{positive $\I$-model structure}. A functor $F$ is fibrant for this model structure if $F(S) \to F(T)$ is a weak equivalence for all $\varnothing \neq S \hookrightarrow T$. By an \emph{$\I$-algebra} we mean a (unital) lax symmetric monoidal functor $\I \to \mathsf{Ch}_R$, i.e.\ a commutative monoid in $\mathsf{Ch}_R^\I$. There is an induced model structure on $\I$-algebras, with $A \to B$ a morphism of $\I$-algebras a fibration (weak equivalence) when $A \to B$ is a fibration (weak equivalence) for the positive $\I$-model structure. Then the category of $\I$-algebras is Quillen equivalent to the category of unital $\mathbb E_\infty$-algebras, and if $A$ is an $\I$-algebra then $\hocolim_\I A$ is an $\mathbb E_\infty$-algebra. 

	 Now the category of non-unital $\mathbb E_\infty$-algebras in $\mathsf{Ch}_R$ is Quillen equivalent to the category of augmented $\mathbb E_\infty$-algebras $A \to R$, where the latter category has its natural model structure as a slice category. By the opposite of \cite[Theorem 1.3.17]{hovey} and the results recalled in \S \ref{richtersagavesummary}, non-unital $\mathbb E_\infty$-algebras are then Quillen equivalent to the category of $\I$-algebras with an augmentation $A \to \underline R$, where $\underline R$ denotes the constant $\I$-algebra at $R$. If $A$ is such an augmented $\I$-algebra then its augmentation ideal $\overline A$, defined by $\overline A(S) = \mathrm{Ker}(A(S) \to R)$, is a well defined $\I_+$-algebra. The functor from augmented $\I$-algebras to $\I_+$-algebras has a left adjoint given by adjoining a unit freely in each arity. This is in fact a Quillen adjunction between augmented $\I$-algebras (with the positive $\I$-model structure) and $\I_+$-algebras (with the model structure of Theorem \ref{ps}). The category of $\I_+$-algebras admits a Bousfield localization in which the weak equivalences are the maps inducing a weak equivalence of homotopy colimits, and the Quillen adjunction between augmented $\I$-algebras and $\I_+$-algebras induces a Quillen equivalence between augmented $\I$-algebras and this Bousfield localization, proving Theorem \ref{bousfield}. The homotopy colimit of an $\I_+$-algebra is the augmentation ideal in the homotopy colimit of the $\I$-algebra obtained by freely adjoining a unit in each arity, proving Theorem \ref{non-unitalthm}. \end{para}

\begin{prop}
	Let $X$ be a paracompact Hausdorff space. The tcdga $R\Gamma^\otimes(X,R)$ has a natural structure of fibrant dg $\I$-algebra. If $X$ is in addition locally compact, then $R\Gamma_c^\otimes(X,R)$ is a fibrant dg $\I_+$-algebra. 
\end{prop}

\begin{proof}
	Recall that $R\Gamma^\otimes(X,R)$ and $R\Gamma_c^\otimes(X,R)$ are obtained by applying the functors $\Gamma$ (resp.\ $\Gamma_c$) to the sheaf of tcdga's on $X$ given by $S \mapsto \mathrm{Gode}(R)^{\otimes S}$. We claim that the construction $S \mapsto \mathrm{Gode}(R)^{\otimes S}$ can be naturally given the structure of a sheaf of $\I$-algebras. Indeed, note that there is an augmentation $R \stackrel \sim \to \mathrm{Gode}(R)$, and that $R$ is the monoidal unit in the category of sheaves of $R$-modules on $X$. By inserting the monoidal unit and applying the augmentation, we obtain natural maps $\mathrm{Gode}(R)^{\otimes S} \stackrel \sim \to \mathrm{Gode}(R)^{\otimes T}$ for every injection $S \hookrightarrow T$. Since both sheaves are soft (resp.\ $c$-soft) when $S$ and $T$ are nonempty --- see the proof of Proposition \ref{rightcohomology} --- the induced map on (compactly supported) global sections is a quasi-isomorphism too, so $R\Gamma^\otimes(X,R)$ and $R\Gamma_c^\otimes(X,R)$ are fibrant with respect to the resulting dg $\I$-algebra (resp.\ dg $\I_+$-algebra) structure.  
\end{proof}

\begin{thm}
	Let $X$ be a CW complex. The tcdga $R\Gamma^\otimes(X,R)$ is weakly equivalent to the ``homotopically constant'' tcdga given by the $\mathbb E_\infty$-algebra of cochains $C^\bullet(X,R)$. If $X$ is the complement of a subcomplex in a finite CW complex then $R\Gamma^\otimes_c(X,R)$ is weakly equivalent to the compactly supported cochains $C^\bullet_c(X,R)$ with its structure of non-unital $\mathbb E_\infty$-algebra.
\end{thm}

\begin{proof}As in Theorem \ref{APL} this follows from Mandell's uniqueness theorem. 
\end{proof}

\begin{rem}
	A direct proof is also possible. Let us just sketch the argument. We take $X$ to be paracompact, Hausdorff and locally contractible. First one shows, using local contractibility, that there is a weak equivalence of sheaves of $\mathbb{E}_\infty$-algebras on $X$ between the constant sheaf $R$ and the sheafifaction of the complex of presheaves $C^\bullet(-,R)$ of singular cochains. Then their derived global sections are also weakly equivalent as $\mathbb E_\infty$-algebras. But the global sections of the sheafification of $C^\bullet(-,R)$ is the quotient
	$$ C^\bullet(X,R)/C_0^\bullet(X,R)$$
	by the subcomplex of singular cochains which are zero on some open cover of $X$, and the quotient map $C^\bullet(X,R)\to C^\bullet(X,R)/C_0^\bullet(X,R)$ is a weak equivalence of $\mathbb E_\infty$-algebras. Also, the sheafification of $C^\bullet(-,R)$ is flabby and its derived global sections are just its global sections. On the other hand we claim that $\hocolim_{\I} R\Gamma^\otimes(X,R)\cong R\Gamma(X,R)$ as $\mathbb E_\infty$-algebras; indeed, this follows from $R\Gamma^\otimes(X,R)$ being a fibrant $\I$-algebra. 
\end{rem}

\begin{rem}
	Richter and Sagave \cite{richtersagave} have recently constructed a dg $\I$-algebra model for the cochains on a space by completely different methods. Our sheaf-theoretically defined functor $R\Gamma^\otimes$ offers an alternative to their construction. After this paper appeared on the arXiv, Chataur and Cirici \cite{chataurcirici} have constructed a functor $\mathbb R_{\EuScript E}\Gamma$ associated to any cofibrant $\mathbb E_\infty$-operad $\EuScript E$, quasi-isomorphic to $R\Gamma$ and taking sheaves of $\EuScript E$-algebras to $\EuScript E$-algebras.
\end{rem}

\section{Combinatorial preliminaries}

\subsection*{Order complex}

\begin{para}
	Throughout this section we let $(P,\preceq)$ denote a finite partially ordered set (poset). We define the \emph{order complex} of $P$, $\nerve P $, to be the simplicial complex whose $p$-simplices are chains $x_0 \prec x_1 \prec \ldots \prec x_p$ of comparable elements of $P$. Equivalently, $\nerve P$ is the geometric realization of the nerve of $P$, when we think of $P$ as a category.
\end{para}

\begin{para}\label{variant-of-order-cpx}If $P$ has either a largest or smallest element, its order complex is contractible. In this case it is common to remove the top and/or bottom element before taking the order complex; however, the following related construction will be more convenient for us. We always use $\hat 0$ and $\hat 1$ to denote a smallest and a largest element of $P$, respectively. Define
	\begin{align*}
	\nerveu{P} &= \nerve P / \nerve{P \setminus \{\hat 1\}}, \\
	\nerved{P} &= \nerve P / \nerve{P \setminus \{\hat 0\}}, \\
	\nerveud{P} &= \nerve P / ( \nerve{P \setminus \{\hat 0\}} \cup \nerve{P \setminus \{\hat 1\}}).
	\end{align*}
	For example, $\nerveud P$ is the based CW complex obtained from $\nerve P$ by collapsing to a point the subcomplex consisting of all simplices not containing both $\hat 0$ and $\hat 1$. We always follow the standard convention of interpreting the quotient space $X/A$ as the pushout $X \leftarrow A \rightarrow \mathrm{pt}$, so that when $A$ is empty we obtain $X$ with a disjoint basepoint. 
\end{para}

\begin{lem}\label{susp}
	If $P$ has a largest element $\hat 1$ or a smallest element $\hat 0$ respectively, then 
	\begin{align*}
	\nerveu P &\cong \Sigma \nerve{P \setminus \{\hat 1\}},\\
	\nerved P &\cong \Sigma \nerve{P \setminus \{\hat 0\}},
	\end{align*}
	where $\Sigma$ denotes the unreduced suspension. If $P$ has both a top and bottom element, then 
	\[\nerveud P \cong \begin{cases}
	S^0 & \text{if } P = \{\ast\}, \\
	\Sigma^2 \nerve {P \setminus \{\hat 0,\hat 1\}} & \text{if } \vert P \vert \geq 2.
	\end{cases}\]
\end{lem}	

\begin{proof}
	If $P$ has a largest element, then $\nerve P \cong C \nerve {P \setminus \{\hat 1\}}$, where $CX$ denotes the cone on a space $X$. Then use that $CX/X \cong \Sigma X$. When $P$ has both a bottom and a top element, apply this argument twice, treating the case when $\hat 0 = \hat 1$ separately. 
\end{proof}

\begin{rem}
	In this paper we will only be concerned with the (reduced) cohomology of order complexes. The reader may reasonably wonder why we bother introducing the constructions $\nerveu P, \nerved P $ and $\nerveud P$, when by the previous lemma their cohomologies can all be expressed in terms of the cohomology of the usual order complex of $P$ with some elements removed. However, in this way we reduce the number of degree shifts involved in the constructions and the main results; moreover, if we didn't have $\nerveud \ast \cong S^0$ then we would have to introduce some rather unnatural conventions to deal with the case of a one-element poset (as was done in \cite[top of p. 2531]{spectralsequencestratification}). We also have the following appealing formula:
\end{rem}

\begin{lem}\label{smash}
	Let $P$ and $Q$ be posets with largest and smallest elements. Then
	$$ \nerveud {P \times Q} \cong \nerveud P \wedge \nerveud Q,$$
	$\wedge$ denoting the smash product.
\end{lem}

\begin{proof}The order complex of a product is homeomorphic to the product of the order complexes. Now one checks that both sides are obtained by collapsing the same subspace of $\nerve {P \times Q}$ to a point. 
\end{proof}

\begin{para}
	We say that $P$ is a \emph{lattice} if every subset of $P$ has a unique least upper bound (join) and greatest lower bound (meet). In particular, $P$ must have a largest and smallest element, which are the empty join and meet respectively. In a finite poset, the existence of joins implies the existence of meets, and vice versa: the meet of a subset can be defined as the join of the set of all its lower bounds. 
\end{para}

\begin{para}
	Suppose that $P$ has a smallest element $\hat 0$. The minimal elements of $P \setminus \{\hat 0\}$ are called \emph{atoms}.
\end{para}

\begin{lem}\label{bjorner}
	Let $P$ be a finite lattice. Let $J_P \subseteq P$ be the sublattice consisting of all elements which are joins of atoms, including the empty join. If $\hat 1 \in P$ does not lie in $J_P$, then $\nerveud P$ is contractible. If $\hat 1$ lies in $J_P$, then $\nerveud {P}$ is homotopy equivalent to $\nerveud {J_P}$.
\end{lem}

\begin{proof}
	It is well known that an adjunction between categories induces a homotopy equivalence of nerves: the unit and counit $\mathbf 1 \Rightarrow GF$ and $FG \Rightarrow \mathbf 1$ give rise to homotopies between the induced maps between nerves, whence the result. In the case at hand we have such an adjunction: we let $J_P \to P$ be the inclusion, and $P \to J_P$ takes an element $x$ to the join of all atoms below $x$. If $\hat 1 \in J_P$ then this restricts to an adjunction between $J_P \setminus \{\hat 0,\hat 1\}$ and $P \setminus \{\hat 0,\hat 1\}$, which together with Lemma \ref{susp} gives the result. If $\hat 1 \notin J_P$ then we get an adjunction between $J_P \setminus \{\hat 0\}$ and $P \setminus \{\hat 0,\hat 1\}$. But $\nerve {J_P \setminus \{\hat 0\}}$ is contractible since $J_P$ has a largest element, so then also $\nerve {P \setminus \{\hat 0,\hat 1\}}$ is contractible, so $\nerveud P$ is contractible. 
\end{proof}
%

%
%
%
%

\subsection*{Sheaf cohomology on posets}

\begin{para}
	We say that a subset $U \subseteq P$ is \emph{upwards closed} if $x \in U$ and $x \preceq y$ implies $y \in U$. We let $\EuScript U(P)$ denote the collection of upwards closed subsets of $P$. The set $\EuScript U(P)$ is itself a finite poset, partially ordered by inclusion. The upwards closed subsets form a topology on $P$, the \emph{Alexandrov topology.}
\end{para}

\begin{para}\label{construction-sheafcohomology}
	Let $\Phi \colon P \to \mathsf{Ch}_R$ be a functor from $P$ to cochain complexes. Let $U \subseteq P$ be upwards closed. We define a double complex
	$$ \mathfrak{B}(U,\Phi) = \bigoplus_{\varnothing \neq C \subseteq U \text{ a chain}} \Phi(\max C),$$
	where the direct sum is taken over all strictly increasing chains of elements in $U$. The vertical grading and vertical differentials are given by the internal grading and differential of $\Phi(x)$. The horizontal grading is given by $\vert C \vert - 1$, and the horizontal differential is an alternating sum over all ways of adding an extra element to the chain; note that if $C \subset C'$, then $\max C \preceq \max C'$, which produces a map $\Phi(\max C) \to \Phi(\max C')$. Equivalently, $\mathfrak{B}(U,\Phi)$ is the bar resolution computing the homotopy limit $\operatorname{holim}_U \Phi$.
\end{para}

\begin{rem}
	For our purposes we will never need the interpretation of $\mathfrak{B}(U,\Phi)$ as a homotopy limit. Nevertheless it may be helpful psychologically to know that it is indeed a homotopy limit, and the reader may at this point object that the bar resolution is usually only used to compute homotopy limits in simplicial model categories. But the model category of unbounded cochain complexes is not simplicial. Let us for completeness argue that this is a homotopy limit. If $\Phi \colon P \to \mathsf{Ch}_R$ is a functor, then define
	$$ R\Phi (x) = \bigoplus_{\substack{\varnothing \neq C \subseteq P \text{ a chain}\\ \min(C) \succeq x}} \Phi(\max(C))$$
	 with a differential as above given by adding elements to the chain. If $x \preceq y$ then we obtain $R\Phi(x) \to R\Phi(y)$ by projecting onto the summands for which $\min(C) \succeq y$.  We get a natural transformation $\Phi \to R\Phi$, mapping $\Phi(x)$ diagonally into the summands given by singleton chains, which is a quasi-isomorphism pointwise (argue as in Proposition \ref{filtrationprop}). Since fibrations in $\mathsf{Ch}_R$ are surjections, it follows that $R\Phi$ is Reedy fibrant (i.e.\ all maps $R\Phi(x) \to R\Phi(y)$ are surjections), so $\Phi \to R\Phi$ is a fibrant replacement. The limit of $R\Phi$ is precisely the bar resolution, and $P$ (being a poset) has cofibrant constants so that the homotopy limit is computed by any Reedy fibrant replacement. 
\end{rem}

\begin{para}\label{construction-sheafcohomology2}Assume now that $P$ has a smallest element $\hat 0$. We define similarly a double complex
	$$  \widetilde{\mathfrak{B}}(U,\Phi) = \bigoplus_{C \subseteq U \text{ a chain} } \Phi(\max C ),$$
	where we now allow also the empty chain; $\max(\varnothing) = \hat 0$ is the minimal element of $P$. Again the vertical differential and grading come from $\Phi$, the horizontal differential is an alternating sum over all ways of adding an element to a chain, but the horizontal grading is given by the cardinality $\vert C \vert$. The two cochain complexes $ \mathfrak{B}(U,\Phi)$ and $\widetilde{\mathfrak{B}}(U,\Phi)$ differ only by an ``augmentation'' by $\Phi(\hat 0)$ and by a degree shift; there is a short exact sequence
	$$ 0 \to \Phi(\hat  0) \to  \widetilde{\mathfrak{B}}(U,\Phi) \to  \mathfrak{B}(U,\Phi)[1] \to 0.$$
\end{para}

\begin{para}
	If $\Phi$ is the constant functor taking each element to the $R$-module $R$, then the complexes $\mathfrak B(U,\Phi)$ and $\widetilde{\mathfrak B}(U,\Phi)$ compute the cohomology (resp. the reduced cohomology) of the order complex $\nerve U$. Indeed, there is an evident isomorphism with the complexes of (reduced) cellular chains of $\nerve U$. 
\end{para}

\begin{prop}
	Let $\EuScript U(P)$ denote the poset of upwards closed subsets in $P$. Then $\mathfrak{B}(-,\Phi)$ and $\widetilde{\mathfrak{B}}(-,\Phi)$ define contravariant functors $\EuScript U(P) \to \mathsf{Ch}_R$. If $\Phi$ and $\Psi$ are quasi-isomorphic as functors $P \to \mathsf{Ch}_R$, then $\mathfrak{B}(-,\Phi)$ is quasi-isomorphic to $\mathfrak{B}(-,\Psi)$, and similarly for $\widetilde{\mathfrak{B}}$.
\end{prop}

\begin{proof}
	This is clear.
\end{proof}

\begin{rem}When $P$ is understood as a topological space with the Alexandrov topology, then $\Phi$ can be seen as defining a complex of sheaves on this topological space. Indeed, a sheaf on a topological space is completely determined by its values on a basis of open sets, and any functor defined on basic open sets which satisfies the sheaf axiom for covers by basic opens sets extends uniquely to a sheaf on arbitrary open sets by Kan extension. Now a basis for the Alexandrov topology is given by open sets of the form $P^{\succeq x} = \{y \in P : y \succeq x\}$, and any functor $P \to \mathsf{Ch}_R$ defines a presheaf of cochain complexes on this basis; moreover, this presheaf will automatically satisfy the sheaf axiom, since any open cover of $P^{\succeq x}$ needs to contain the whole open set $P^{\succeq x}$ as one of its elements. The complex $\mathfrak B(U,\Phi)$ computes the cohomology of this complex of sheaves over the open set $U$ --- indeed, sheaf cohomology on a space $X$ can be defined as the homotopy limit of the values of the sheaf over the category of all open subsets of $X$. 
	However, this perspective will not really be used in this paper.

\end{rem}

\subsection*{Two spectral sequences}

\begin{para}
	If $x \in P$, then we denote by $P^{\preceq x}$ the the lower interval $\{y \in P : y \preceq x\}$. 
\end{para}

\begin{para}\label{filtration}Let $P, U$ and $\Phi$ be as in \S \ref{construction-sheafcohomology} and \S \ref{construction-sheafcohomology2}. Suppose moreover that we are given a strictly increasing function $\rho \colon P \to \mathbf Z$. We may define decreasing filtrations of the complexes $\mathfrak B(U,\Phi)$ and $\widetilde{\mathfrak B}(U,\Phi)$ as 
	$$F^p \mathfrak{B}(U,\Phi) = \bigoplus_{\substack{\varnothing \neq C \subseteq U \text{ a chain} \\ \rho(\max(C)) \geq p}} \Phi(\max C)$$
	and
	$$F^p \widetilde{\mathfrak{B}}(U,\Phi) = \bigoplus_{\substack{C \subseteq U \text{ a chain} \\ \rho(\max(C)) \geq p}} \Phi(\max C).$$
\end{para}

\begin{prop}\label{filtrationprop}The above filtration on $\mathfrak B(U,\Phi)$ has the property that $$ H^\bullet(\mathrm{Gr}^p {\mathfrak{B}}(U,\Phi)) = \bigoplus_{\substack{x \in U \\ \rho(x) = p}} \widetilde H^\bullet(\nerveu {U^{\preceq x}}, \Phi(x)).$$For $U \subsetneq P$, let $U_0 = U \cup \{\hat 0\}$. Then
	$$ H^\bullet(\mathrm{Gr}^p \widetilde{\mathfrak{B}}(U,\Phi)) = \bigoplus_{\substack{x \in U_0 \\ \rho(x) = p}} \widetilde H^\bullet(\nerveud {U_0^{\preceq x}}, \Phi(x)).$$
	
\end{prop}
\begin{proof}
	Clearly we have
	$$ \mathrm{Gr}^p \mathfrak B(U,\Phi) = \bigoplus_{\substack{\varnothing \neq C \subseteq U \text{ a chain} \\ \rho(\max C) = p}} \Phi(\max C).$$
	If two chains in the complex have different maximal elements, then the two chains are incomparable. (This is where we use that $\rho$ was assumed strictly increasing.) Thus this complex becomes a direct sum of complexes indexed by the possible maximal elements, i.e.\ the elements $x \in U$ with $\rho(x)=p$. The summand corresponding to $x$ is given by $$\bigoplus_{\substack{C \subseteq U \text{ a chain} \\ \max C = x}} \Phi(x).$$ But the chains in $U$ whose largest element equal $x$ can be identified with the reduced cellular cochains of $\nerveu {U^{\preceq x}}$, so this summand is equal to $\widetilde C^\bullet(\nerveu {U^{\preceq x}}) \otimes \Phi(x)$.
	The cohomology of this complex is $\widetilde H^\bullet(\nerveu {U^{\preceq x}}, \Phi(x))$.
	
	The argument for $\widetilde{\mathfrak B}$ is similar.
\end{proof}

\begin{prop}\label{prop-sseq0} Under the above assumptions, there are spectral sequences
	$$ E_1^{pq} =  \bigoplus_{\substack{x \in U \\ \rho(x) = p}} \widetilde H^{p+q}(\nerveu {U^{\preceq x}}; \Phi(x)) \implies H^{p+q}(\mathfrak B(U,\Phi))$$
	and 
	$$ E_1^{pq} =  \bigoplus_{\substack{x \in U_0 \\ \rho(x) = p}} \widetilde H^{p+q}(\nerveud {U_0^{\preceq x}}; \Phi(x)) \implies H^{p+q}(\widetilde{\mathfrak B}(U,\Phi)).$$
\end{prop}

\begin{proof}
	This is the spectral sequence of a filtered complex, applied to the filtrations on $\mathfrak B$ and $\widetilde{\mathfrak B}$.
\end{proof}
\begin{para}
	
	These spectral sequences, in the context of sheaf cohomology on posets, were first considered by Bac{\l}awski \cite[Section 4]{baclawski}.
\end{para}

\begin{prop}\label{prop-sseq}Let $P, U$ and $\Phi$ be as in \S \ref{construction-sheafcohomology} and \S \ref{construction-sheafcohomology2}, and assume that $P$ is a lattice. Let $\rho \colon P \to \Z$ be strictly increasing. Let $U \subsetneq P$ be upwards closed and let $J_U$ be subposet consisting of all joins of minimal elements of $U$. Let $U_0 = U \cup \{\hat 0\}$ and $J_{U_0} = J_U \cup \{\hat 0\}$. The spectral sequences of Proposition \ref{prop-sseq0} may be simplified to
	$$ E_1^{pq} =  \bigoplus_{\substack{x \in J_U \\ \rho(x) = p}} \widetilde H^{p+q}(\nerveu {J_U^{\preceq x}}; \Phi(x)) \implies H^{p+q}(\mathfrak B(U,\Phi))$$
	and 
	$$ E_1^{pq} =  \bigoplus_{\substack{x \in J_{U_0} \\ \rho(x) = p}} \widetilde H^{p+q}(\nerveud {J_{U_0}^{\preceq x}}; \Phi(x)) \implies H^{p+q}(\widetilde{\mathfrak B}(U,\Phi)).$$
\end{prop}

\begin{proof}
	Let $x \in U_0$. Lemma \ref{bjorner} applied to the lattice $U_0^{\preceq x}$ shows that $\nerveud {U_0^{\preceq x}}$ is contractible if $x \notin J_{U_0}$, so in the second spectral sequence of Proposition \ref{prop-sseq0} we may remove all terms except those with $x \in J_{U_0}$. When $x \in J_{U_0}$, Lemma \ref{bjorner} also implies that $\nerveud {U_0^{\preceq x}} \simeq \nerveud {J_{U_0}^{\preceq x}}$. The result follows. The argument for the first spectral sequence is analogous.
\end{proof}

\subsection*{Semidirect product}

\begin{para}\label{semidirect}
	Let $G$ be a group acting on $P$. Define the \emph{semidirect product} $G \ltimes P$ as the following category: the objects of $G \ltimes P$ are the elements of $P$, and a morphism $x \to y$ is an element $g \in G$ such that $g \cdot x \preceq y$. Composition of morphisms is given by group multiplication. This is an instance of the Grothendieck construction, when we think of $P$ as a category and the action of $G$ on $P$ as defining a functor $G \to \mathsf{Cat}$. Note that there is a functor $G \ltimes P \to G$, and that $P$ sits inside $G \ltimes P$ as the subcategory consisting of all morphisms given by the identity element in $G$. Hence we get a ``short exact sequence'' of sorts,
	$$ P \to G \ltimes P \to G.$$
\end{para}

\begin{para}
	Suppose that $\Phi$ is a functor $G \ltimes P \to \mathsf{Ch}_R$. We may in particular restrict $\Phi$ to $P$, which allows us to make sense of $\mathfrak B(U,\Phi)$ and $\widetilde{\mathfrak B}(U,\Phi)$ for $U \subseteq P$ upwards closed. However, the action of $G$ furnishes an additional functoriality: if $g \cdot U = V$, for $U, V \subseteq P$ upwards closed, then multiplication by $g$ induces a map $\mathfrak B(U,\Phi) \to \mathfrak B(V,\Phi)$. Indeed, if $g \cdot x = y$ then $g$ gives a morphism $x \to y $ in $G \ltimes P$, so $\Phi(g)$ is a morphism $\Phi(x) \to \Phi(y)$; we use this to map each summand of $\mathfrak B(U,\Phi)$  to a corresponding summand of $\mathfrak B(V,\Phi)$. We record this as a proposition:
\end{para}

\begin{prop}
	If we are given $\Phi \colon G \ltimes P \to \mathsf{Ch}_R$, then $\mathfrak B(-,\Phi)$ and $\widetilde{\mathfrak B}(-,\Phi)$ define contravariant functors $G \ltimes \EuScript U(P) \to \mathsf{Ch}_R$. In particular, if $G_U \subseteq G$ is the subgroup of elements $g$ such that $g \cdot U = U$, then the cochain complexes $\mathfrak B(U,\Phi)$ and $\widetilde{\mathfrak B}(U,\Phi)$ are naturally representations of $G_U$. The spectral sequences of Proposition \ref{prop-sseq} are similarly equivariant. 
\end{prop}

\section{Cohomology of configuration spaces}

\subsection*{Sheaves on the partition lattice}

\begin{para}
	Let $S$ be a finite set. Let $\Pi_S$ denote the poset consisting of all equivalence relations on $S$, ordered by refinement: we say that $\sim \,\, \preceq \,\, \sim'$ if $x \sim y$ implies $x \sim' y$. The poset $\Pi_S$ is a lattice: the join of $\sim$ and $\sim'$ is the relation $\sim''$ defined by $x \sim'' y \iff x \sim y$ or $x \sim' y$; the meet of $\sim$ and $\sim'$ is given by  $x \sim'' y \iff x \sim y$ and $x \sim' y$. We call $\Pi_S$ the \emph{partition lattice}. When $S = \{1,\ldots,n\}$ we denote it simply $\Pi_n$.
\end{para}

\begin{para}We usually prefer to think of the elements of $\Pi_S$ as partitions of $S$ into nonempty disjoint subsets. The smallest element is the partition of $S$ into $1$-element blocks, and the largest element is the partition of $S$ into a single block. We adopt the somewhat abusive notation of denoting a typical partition $S = \coprod_{i\in I} T_i$ by the symbol $T$. 
\end{para}

\begin{para}\label{functor-from-tcdga}
	Let $A$ be a tcdga (\S \ref{definition-tcdga}). Then we obtain a functor $\Phi_A \colon \Pi_S \to \mathsf{Ch}_R$, for any finite set $S$. Specifically, if $T$  denotes the partition $ S= \coprod_{i\in I} T_i$, then
	$$ \Phi_A(T) = \bigotimes_{i\in I}A(T_i).$$
	If $T \preceq T'$ then $T'$ is obtained from $T$ by merging some blocks of the partition. The map $\Phi_A(T) \to \Phi_A(T')$ is given by applying the maps $A(T_i) \otimes A(T_j) \to A(T_i \sqcup T_j)$ furnished by the tcdga structure. 
\end{para}

\begin{para}\label{functor-from-shuffle-algebra}
	In fact to define a functor $\Pi_S \to \mathsf{Ch}_R$ we do not need the full data of a tcdga. Let us instead suppose that $S$ is totally ordered. Then every subset of $S$ inherits a natural total order as well, so that for each partition $\coprod_{i\in I} T_i$ of $S$ we get an induced total ordering on each of the blocks $T_i$. If $A$ is a commutative dg shuffle algebra (\S \ref{defn-shuffle}) then we may similarly define 
	$$ \Phi_A(T) =\bigotimes_{i\in I}A(T_i),$$
	where the maps $\Phi_A(T) \to \Phi_A(T')$ are now given by merging blocks using the shuffle algebra structure on $A$. 
\end{para}

\begin{para}
	Note that if $A$ and $A'$ are quasi-isomorphic as tcdga's (or as commutative dg shuffle algebras), then $\Phi_A$ and $\Phi_{A'}$ are quasi-isomorphic as functors. 
\end{para}

\begin{para}
	The difference between the constructions of \S\ref{functor-from-tcdga} and \S\ref{functor-from-shuffle-algebra} is that when $A$ is a tcdga we will not only get a functor $\Pi_n \to \mathsf{Ch}_R$ for all $n$, but in fact a functor
	$$ \mathbb S_n \ltimes \Pi_n \to \mathsf{Ch}_R$$
	from the semidirect product (\S \ref{semidirect}) of $\Pi_n$ with $\mathbb S_n$, where $\mathbb S_n$ acts in the obvious way on the partition lattice. This will have the consequence that the various spectral sequences we write down for the cohomology of configuration spaces are $\mathbb S_n$-equivariant. 
\end{para}

\subsection*{The functors $\CF$ and $\CD$, and the connection with configuration spaces}

\begin{defn}\label{defn-cf-cd}
	Let $U \subset \Pi_n$ be upwards closed, and let $A$ be a tcdga. We define
	\begin{align*}
	\CD(U,A) &= \mathfrak B(U,\Phi_A) \\
	\CF(U,A) &= \widetilde {\mathfrak B}(U,\Phi_A),
	\end{align*} 
	where $\mathfrak B$ and $ \widetilde {\mathfrak B}$ are defined in \S \ref{construction-sheafcohomology} and \S \ref{construction-sheafcohomology2}, and $\Phi_A$ is defined in \S \ref{functor-from-tcdga}. 
\end{defn}

\begin{rem}
	We observe that $\CD$ and $\CF$ are functors
	$$ (\mathbb S_n \ltimes \EuScript U(\Pi_n)) \times \mathsf{TCDGA}_R \to \mathsf{Ch}_R.$$ Note also that quasi-isomorphic tcdga's give rise to quasi-isomorphic functors. Finally we remark that if $A$ is only a commutative shuffle dg algebra, then we can define $\CD$ and $\CF$ similarly; the resulting functors are just not defined on the semidirect product with $\mathbb S_n$. 
\end{rem}

\begin{para}Let $X$ be a locally compact Hausdorff space space and let $\EuScript F$ be a complex of sheaves on $X^n$. For $T \in \Pi_n$, let $j(T) \colon X(T) \to X^n$ be the corresponding locally closed subset of $X^n$, and let $i(T) \colon \overline{X(T)} \to X^n$ be the inclusion of its closure. Note that each closed subset $\overline {X(T)}$ is a cartesian product of $X$ with itself: if $T$ has blocks $T_1,\ldots,T_k$, then $\overline {X(T)} \cong X^k$. Let $U\subseteq \Pi_n$ be upwards closed, $j \colon F(X,U) \to X^n$ the open inclusion, and $i \colon D(X,U) \to X^n$ its closed complement. The following result is a special case of \cite[Proposition 3.1]{spectralsequencestratification}, see also the results of \cite{getzler99} when $F(X,U)=F(X,n)$. \end{para}

\begin{prop}\label{proposition-petersen}
	There are quasi-isomorphisms
	$$ j_!j^\ast \EuScript F \simeq \bigoplus_{C \subseteq U \text{ \rm a chain}} i({\max C})_\ast i({\max C})^\ast \EuScript F$$
	and 
	$$ i_\ast i^\ast \EuScript F \simeq \bigoplus_{\varnothing \neq C \subseteq U \text{ \rm a chain}} i({\max C})_\ast i({\max C})^\ast \EuScript F$$
	where the right hand sides are considered as double complexes of sheaves: the bigradings and the differentials on the right hand side are given by exactly the same formulas as for the complexes $\CF$ and $\CD$, see \S\S \ref{construction-sheafcohomology}---\ref{construction-sheafcohomology2}.
\end{prop}

\begin{proof}Consider $X^n$ as a stratified space with a single open stratum given by $F(X,U)$; the remaining strata are the locally closed subsets $X(T) \subset X^n$ where $T \in U$. Then the poset of strata is given by $U \cup \{\hat 0\}$. In \cite[Section 3]{spectralsequencestratification}, a general construction is explained for resolving sheaves of the form $j_!j^\ast \EuScript F$ in this situation, and the double complex $\bigoplus_{C \subseteq U \text{ \rm a chain}} i({\max C})_\ast i({\max C})^\ast \EuScript F$ is exactly equal to the complex $L^\bullet(\EuScript F)$ defined in loc.\ cit.

%
	Now there is also a distinguished triangle
	$$ j_!j^\ast \EuScript F \to \EuScript F \to i_\ast i^\ast \EuScript F,$$
	and the double complex $\bigoplus_{C \subseteq U \text{ \rm a chain}} (i^{\max C})_\ast(i^{\max C})^\ast \EuScript F$ differs only by a shift in grading and a coaugmentation by the additional summand $\EuScript F$, corresponding to the empty chain. The result follows. \end{proof}

\begin{para}Since the proof is short, let us outline the proof of the quasi-isomorphism $j_!j^\ast \EuScript F \to L^\bullet(\EuScript F)$ from \cite{spectralsequencestratification}. Consider the summand of $\bigoplus_{C \subseteq U \text{ \rm a chain}} i({\max C})_\ast i({\max C})^\ast \EuScript F$ corresponding to the empty chain; this gives a copy of $\EuScript F$. The natural map $j_!j^\ast \EuScript F \to \EuScript F$ provides a map of complexes 
	$$j_!j^\ast \EuScript F \to \bigoplus_{C \subseteq U \text{ \rm a chain}} i({\max C})_\ast i({\max C})^\ast \EuScript F.$$We may consider this as an augmented complex and it is enough to prove that it is acyclic. This can be checked on stalks. If $x \in F(X,U)$ then the stalk of both sides at $x$ is just $\EuScript F_x$; note that all summands except the empty chain on the right hand side have support outside $F(X,U)$. If $x$ lies in some non-open stratum $X(T)$, then the stalk of the left hand side is zero and the stalk of the right hand side is given by the tensor product
	$$ \EuScript F_x \otimes \widetilde C^\bullet(\nerved {U_0^{\preceq T}}),$$
	where $U_0 = U \cup \{\hat 0\}$. But this order complex is contractible since $ {U_0^{\preceq T}}$ has a largest element.
\end{para}

\begin{thm}\label{mainthm}
	Let $X$ be a paracompact Hausdorff space, $\EuScript F$ a complex of sheaves of $R$-modules on $X$, $R$ a ring of finite global dimension. Let $U \subset \Pi_n$ be upwards closed. There are quasi-isomorphisms of cochain complexes
	\begin{align*}
	\CF(U,R\Gamma^\otimes(X,\EuScript F)) & \simeq  C^\bullet(X^n, D(X,U);\EuScript F^{\boxtimes n}), \\
	\CD(U,R\Gamma^\otimes(X,\EuScript F)) & \simeq  C^\bullet(D(X,U),\EuScript F^{\boxtimes n}),
	\end{align*} 
	which are natural in $X$, $\EuScript F$ and $U$. If $X$ is in additional locally compact, then 
	\begin{align*}
	\CF(U,R\Gamma_c^\otimes(X,\EuScript F)) & \simeq C^\bullet_c(F(X,U),\EuScript F^{\boxtimes n}), \\
	\CD(U,R\Gamma_c^\otimes(X,\EuScript F)) & \simeq C_c^\bullet(D(X,U),\EuScript F^{\boxtimes n}). 
	\end{align*} 
\end{thm}

\begin{proof}Let $\EuScript L$ be the flat and flabby resolution of $\EuScript F$ constructed in Lemma \ref{flabbyflat}. Apply Proposition \ref{proposition-petersen} to get quasi-isomorphisms
	$$ j_!j^\ast \EuScript F^{\boxtimes n} \simeq \bigoplus_{C \subseteq U \text{ \rm a chain}} i({\max C})_\ast i({\max C})^\ast \EuScript L^{\boxtimes n} $$
	and 
	$$ i_\ast i^\ast \EuScript F^{\boxtimes n} \simeq \bigoplus_{\varnothing \neq C \subseteq U \text{ \rm a chain}} i({\max C})_\ast i({\max C})^\ast \EuScript L^{\boxtimes n}.$$
Note now that the right hand sides are complexes of soft sheaves (resp. $c$-soft if $X$ is locally compact), so these resolutions can be used to compute the cohomology of $j_!j^\ast \EuScript F$ and $i_\ast i^\ast$. 

Now the key observation is that the complex of global sections of the complex of sheaves $$ \bigoplus_{C \subseteq U \text{ \rm a chain}} i({\max C})_\ast i({\max C})^\ast \EuScript L^{\boxtimes n}$$ is \emph{isomorphic} to $\CF(U,R\Gamma^\otimes(X,\EuScript F))$. Indeed, the summands in the double complex given by $\CF(U,R\Gamma^\otimes(X,\EuScript F))$ correspond in exactly the same way to chains in the upwards closed subset $U$, and each summand is given precisely by global sections of tensor powers of the resolution $\EuScript L$. Moreover, general sheaf theory tells us that there is a quasi-isomorphism $R\Gamma(X^n,j_!j^\ast \EuScript F^{\boxtimes n}) \simeq C^\bullet(X^n,D(X,U);\EuScript F^{\boxtimes n})$; this proves the first quasi-isomorphism of the theorem. Similarly the compactly supported global sections of the same complex of sheaves is isomorphic to $\CF(U,R\Gamma_c^\otimes(X,\EuScript F))$, and general sheaf theory also tells us that there is a quasi-isomorphism $R\Gamma_c(X^n,j_!j^\ast \EuScript F^{\boxtimes n}) \simeq C^\bullet_c(F(X,U),\EuScript F^{\boxtimes n})$, etc. This finishes the proof. \end{proof}

\begin{para}
	Let us record some special cases of Theorem \ref{mainthm}.
\end{para}

\begin{cor}\label{cor-cdga}
	Let $X$ be a paracompact and locally compact Hausdorff space. Let $A$ be a cdga model for the compactly supported cochains $C^\bullet_c(X,\Q)$. If $A$ is considered as a constant tcdga, then $\CF(U,A) \simeq H^\bullet_c(F(X,U),\Q)$. The functor $$\mathbb S_n \ltimes \EuScript U(\Pi_n) \to \mathsf{Ch}_R$$ which assigns to $U \subset \Pi_n$ upwards closed the compactly supported cohomology $H^\bullet_c(F(X,U),\Q)$ depends up to quasi-isomorphism only on a cdga model of algebra $C^\bullet_c(X,\Q)$.  
\end{cor}

\begin{cor}
	Let $X$ be a paracompact and locally compact Hausdorff space. Suppose that the compactly supported cochain algebra $C^\bullet_c(X,Q)$ is formal. Then the functor $U \mapsto H^\bullet_c(F(X,U),\Q)$ depends only on the cohomology ring $H^\bullet_c(X,\Q)$.  
	\end{cor}

\begin{cor}
	Let $X$ be an oriented manifold. The functor $U \mapsto H_\bullet(F(X,U),\Z)$ depends up to quasi-isomorphism only on the non-unital $\mathbb E_\infty$-algebra structure on $C^\bullet_c(X,\Z)$; equivalently, it depends only on the intersection product on the singular chains of $X$.  
\end{cor}

\begin{rem}\label{algebraic}
	The fact that our methods are purely sheaf-theoretic means in particular that the results can be applied equally well in algebro-geometric settings, with only minor modifications. For example, if $X$ is an algebraic variety and $\EuScript F$ is a complex of $\ell$-adic \'etale sheaves on $X$, then we can define tcdga's $R\Gamma^\otimes(X,\EuScript F)$ and $R\Gamma^\otimes_c(X,\EuScript F)$ by the same procedure as in Section \ref{section:commutativecochains}. Namely, flat resolutions exist in any ringed topos, and the Godement resolution exists in any topos with enough points; the appropriate topos in this case is the pro-\'etale site \cite{bhattscholzeproetale}. With $\Q_\ell$-coefficients one can also define the functors $R\Gamma^{TW}$ and $R\Gamma_c^{TW}$ as in Theorem \ref{navarro-aznar-theorem}. The proofs of Proposition \ref{proposition-petersen} and Theorem \ref{mainthm} go through with no changes. The spectral sequences obtained from our constructions (as we will explain shortly) will in the algebraic case be spectral sequences of $\ell$-adic Galois representations.
	
	However, let us point out that in defining the ``resolution'' $L^\bullet(\EuScript F)$ used in the proof of Proposition \ref{proposition-petersen} we do use crucially that the functors $i^\ast$ and $i_\ast$ are $t$-exact. So we do need slightly more than just a naked six functors formalism; we also need a $t$-structure with expected properties.
\end{rem}

\subsection*{Spectral sequences}

\begin{thm}\label{sseq-thm}
	Let $X$, $\EuScript F$ and $U$ be as in Theorem \ref{mainthm}. Each of the four cases of Theorem \ref{mainthm} produces a spectral sequence, which can be written as
	$$ E_1^{pq} =  \bigoplus_{\substack{T \in {J_{U_0}} \\ \vert T \vert = n-p}} \bigoplus_{i+j=p+q}\widetilde H^i(\nerveud {J_{U_0}^{\preceq T}}; H^j(\overline{X(T)},i(T)^\ast\EuScript F^{\boxtimes n})) \implies H^{p+q}(X^n,D(X,U);\EuScript F^{\boxtimes n})$$
	$$ E_1^{pq} =  \bigoplus_{\substack{T \in J_U \\ \vert T \vert = n-p}} \bigoplus_{i+j=p+q}\widetilde H^i(\nerveu {J_U^{\preceq T}}; H^j(\overline{X(T)},i(T)^\ast\EuScript F^{\boxtimes n})) \implies H^{p+q}(D(X,U),\EuScript F^{\boxtimes n})$$
	$$ E_1^{pq} =  \bigoplus_{\substack{T \in J_{U_0} \\ \vert T \vert = n-p}} \bigoplus_{i+j=p+q}\widetilde H^i(\nerveud {J_{U_0}^{\preceq T}}; H^j_c(\overline{X(T)},i(T)^\ast\EuScript F^{\boxtimes n})) \implies H^{p+q}_c(F(X,U),\EuScript F^{\boxtimes n})$$
	$$ E_1^{pq} =  \bigoplus_{\substack{T \in J_U \\ \vert T \vert = n-p}} \bigoplus_{i+j=p+q}\widetilde H^i(\nerveu {J_U^{\preceq T}}; H^j_c(\overline{X(T)},i(T)^\ast\EuScript F^{\boxtimes n})) \implies H^{p+q}_c(D(X,U),\EuScript F^{\boxtimes n})$$
	respectively. Here $\vert T \vert$ denotes the number of blocks in the partition $T$.
\end{thm}

\begin{proof}
	This is Proposition \ref{prop-sseq} specialized to the current situation.
\end{proof}

\begin{rem}Suppose that the partition $T$ has blocks $T_1,\ldots,T_k$ of size $n_1,\ldots,n_k$ respectively. Then $\overline {X(T)} \cong X^k$, and the sheaf $i(T)^\ast \EuScript F^{\boxtimes n}$ is given by 
	$$ \EuScript F^{\otimes n_1} \boxtimes \ldots \boxtimes \EuScript F^{\otimes n_k}.$$
	Hence in all cases the $E_1$-differential of the spectral sequence depends only on: (a) the reduced cohomology of lower intervals in the poset $J_U$, (b) the (compactly supported) cohomology of $X$ with coefficients in the tensor powers of the sheaf $\EuScript F$. The $E_1$-differential is given by multiplication in the twisted commutative algebra given by the direct sum
	$$ \bigoplus_{n \geq 1} H^\bullet(X,\EuScript F^{\otimes n}),$$
	i.e. the cohomology of the tcdga $R\Gamma^{\otimes}(X,\EuScript F)$. 
 \end{rem}

\subsection*{Comparison with the spectral sequence of Cohen--Taylor--Totaro}

\begin{para}\label{comparisonstart}Suppose that $M$ is an oriented $d$-dimensional manifold. By Poincar\'e duality there is a natural isomorphism
$$ H_{d-k}(M,\Z)\cong H^k_c(M,\Z)  $$
relating homology to cohomology with compact supports. Thus if $X$ is an oriented manifold, the third spectral sequence of Theorem \ref{sseq-thm} yields a spectral sequence computing the homology of the configuration space $F(X,n)$. We claim that the resuling spectral sequence is the \emph{dual} of the Cohen--Taylor--Totaro spectral sequence \cite{cohentaylor,totaro} computing the \emph{cohomology} of $F(X,n)$. More precisely, Totaro constructed this spectral sequence simply as the Leray spectral sequence in sheaf cohomology associated to the inclusion $F(X,n) \hookrightarrow X^n$, and our claim is that the third spectral sequence of Theorem \ref{sseq-thm} coincides with the \emph{homological} Leray spectral sequence of $F(X,n) \hookrightarrow X^n$, up to reindexing and a degree shift. 
\end{para}

\begin{para}
	Before proving this claim let us recall how the Leray spectral sequence is constructed. Let $Y$ be a space, $f \colon Y \to \mathrm{pt}$ the map to a point. Then $H^\bullet(Y,\Z)$ is the cohomology of the complex $Rf_\ast\Z$ in the derived category $D^+(\mathrm{pt})$, where $\Z$ denotes the constant sheaf on $Y$ (which might be more naturally denoted $f^\ast \Z$). Any map from $Y$ to another space $Z$ gives us a factorization $$Y \stackrel g\longrightarrow Z \stackrel h\longrightarrow \mathrm{pt},$$
	so that we can write
	$ Rf_\ast \Z \simeq Rh_\ast Rg_\ast \Z$. Now the object $Rg_\ast \Z$ of $D^+(Z)$ has a canonical filtration induced by the $t$-structure on $D^+(Z)$, which is uniquely determined in the filtered derived category by the fact that its $q$th graded summand is isomorphic to $R^qg_\ast \Z[-q]$. After applying $Rh_\ast$ we obtain also a filtration of $Rf_\ast \Z$, and the spectral sequence of this filtered complex is the Leray spectral sequence $E_2^{pq} = H^p(Z,R^q g_\ast \Z) \implies H^{p+q}(Y,\Z)$. If the spaces involved are of finite type then the homology of $Y$ is given by the Verdier dual $\mathbb D Rf_\ast \Z \simeq Rf_! \mathbb D \Z$, and the filtration of $\mathbb D Rf_\ast \Z \simeq \mathbb D Rh_\ast Rg_\ast \Z$ induced by the filtration of $Rg_\ast \Z$ is called the homological Leray spectral sequence. 
\end{para}

\begin{para}\label{comparison}
	We will also need to recall the relationship between the cohomology of the partition lattice and the cohomology of configuration spaces of points in Euclidean space. Consider the configuration space $F(\R^d,n)$. Its cohomology is nonzero only in degrees divisible by $d-1$, and there is an isomorphism
	$$ H^{q(d-1)}(F(\R^d,n),\Z) \cong \bigoplus_{\substack{T \in \Pi_n \\ \vert T \vert = n-q}} H^q(\nerveud{\Pi_n^{\preceq T}},\Z),$$
	which for example follows from the Goresky--MacPherson formula \cite{goreskymacphersonstratifiedmorsetheory}, see also \cite[Lemma 3]{totaro}.
\end{para}

\begin{para}\label{comparisonend}
	Let us now assume that $X$ is an oriented $d$-dimensional manifold. Verdier duality on $X$ furnishes an isomorphism $\mathbb D \Z \cong \Z[d]$ in the derived category, where $\Z$ denotes the constant sheaf on $X$. Note that this implies in particular the Poincar\'e duality isomorphism between the homology and the compact support cohomology of $X$, since
	$$ \mathbb D Rf_\ast \Z \simeq Rf_! \mathbb D \Z \simeq Rf_! \Z[d],$$
	where $f$ now denotes the map from $X$ to a point: indeed, the cohomology of the left-hand complex is $H_{-\bullet}(X,\Z)$, and the cohomology of the right-hand complex is $H^{\bullet+d}_c(X,\Z)$. 
	
	Let $j \colon F(X,n) \hookrightarrow X^n$ be the inclusion, and $a \colon X^n \to \mathrm{pt}$ the projection. The homological Leray spectral sequence associated to $j$ is given by filtering $\mathbb D Ra_\ast Rj_\ast \Z \simeq Ra_! \mathbb D Rj_\ast \Z$ via the canonical filtration of $Rj_\ast \Z$. The spectral sequence of Theorem \ref{sseq-thm}, in the case $\EuScript F = \Z$, is given by filtering $Ra_! Rj_! \Z$ by equipping the ``resolution'' of $Rj_! \Z$ from Proposition \ref{proposition-petersen} with the filtration described in \S \ref{filtration}. Verdier duality gives $\mathbb D Rj_\ast \Z \simeq Rj_! \mathbb D\Z \simeq Rj_! \mathbb \Z[nd]$, so our claim will follow once we have verified that the two filtrations of $Rj_! \Z$ coincide up to reindexing; equivalently, that the filtration on $Rj_\ast \Z$ obtained by dualizing the filtration on $Rj_!\Z$ given by the resolution of Proposition \ref{proposition-petersen} coincides with the canonical filtration coming from the $t$-structure of $D^+(X^n)$. We see from Proposition \ref{filtrationprop} that
	$$ \operatorname{Gr}^q Rj_!\Z \cong \bigoplus_{\substack{T \in \Pi_n \\ \vert T \vert=n-q}} \widetilde H^q(\nerveud{\Pi_n^{\preceq T}}, \Z) \otimes i(T)_\ast \Z [-q]$$
	where we have used that the partition lattice is Cohen--Macaulay, which implies that $\nerveud{\Pi_n^{\preceq T}}$ is a wedge of spheres of dimension $q$. Hence
	$$ \mathbb D\operatorname{Gr}^q Rj_!\Z[nd] \cong \operatorname{Gr}^q \mathbb D Rj_!\Z[nd] \cong \bigoplus_{\substack{T \in \Pi_n \\ \vert T \vert=n-q}} \widetilde H^q(\nerveud{\Pi_n^{\preceq T}}, \Z) \otimes i(T)_\ast \Z[-q(d-1)], $$
	where in the last step we used $\mathbb D i(T)_\ast \Z[nd] \cong i(T)_\ast\Z[-qd]$, since $\overline{X(T)}$ is an oriented manifold of dimension $(n-q)d$. 
	But it also follows from \cite[Proof of Theorem 1]{totaro} that there is an isomorphism
	$$ R^{q(d-1)} j_\ast \Z \cong \bigoplus_{\substack{T \in \Pi_n \\ \vert T \vert = n-q}} \widetilde H^q(\nerveud{\Pi_n^{\preceq T}},\Z) \otimes i(T)_\ast \Z.$$
	It follows that the two filtrations on $Rj_\ast\Z \simeq \mathbb D Rj_!\Z[nd]$ are isomorphic in the filtered derived category up to a reindexing given by $q \leftrightarrow q(d-1)$, since the canonical filtration of a complex with respect to a $t$-structure is uniquely determined by its associated graded object. Hence our spectral sequence is the homological version of Totaro's spectral sequence. In particular, this argument does not just identify the first nontrivial pages of the two spectral sequences; it implies an isomorphism of all pages of the two spectral sequences. \end{para}

\section{Configuration spaces of points on $i$-acyclic spaces}\label{section:arabia}

\subsection*{Recovering results of Arabia} 

\begin{para}
	Let us begin by recalling the notion of $i$-acyclicity, which was recently introduced by Arabia \cite{arabia}. When a space $X$ is $i$-acyclic, the compact support cohomology of the configuration space of points on $X$ depends on the cohomology of $X$ itself in the simplest way possible. Although Arabia did not phrase the result in this way, we will see that $i$-acyclic spaces have the property that the two compactly supported spectral sequences of Theorem \ref{sseq-thm} degenerate immediately. 
\end{para}

\begin{defn}
	A topological space $X$ is said to be \emph{i-acyclic over $R$} if $H^\bullet_c(X,R) \to H^\bullet(X,R)$ is the zero map. 
\end{defn}

\begin{examplex}\label{acyclic-example}Let us give some examples of $i$-acyclic spaces \cite[Proposition 1.2.4]{arabia}: 
	\begin{itemize}
		\item Euclidean space is $i$-acyclic over any ring. This can be seen as the special case of either of the next two examples:
		\item If $X$ is an oriented manifold, then $X$ is $i$-acyclic over $R$ if and only if the intersection product on $H_\bullet(X,R)$ vanishes.
		
		\item If $X$ is noncompact and has the cohomology of a point over $R$, then $X$ is $i$-acyclic. 
		\item Open subsets of $i$-acyclic spaces are $i$-acyclic.
		\item If $X$ is $i$-acyclic, and  either $H^\bullet_c(X,R)$ or $H^\bullet_c(Y,R)$ is a projective $R$-module, then $X \times Y$ is $i$-acyclic. For example, a noncompact Lie group is topologically the product of a maximal compact subgroup with Euclidean space, and is therefore $i$-acyclic over any ring. 
		\item If $X$ is $i$-acyclic over $R$ and $G$ is a finite group of automorphisms of $X$ whose order is invertible in $R$, then $X/G$ is $i$-acyclic over $R$.
	\end{itemize}
\end{examplex}

\begin{thm}[(Arabia)]\label{arabiathm} Let $X$ be an $i$-acyclic metrisable $\sigma$-compact and locally compact space over a field $k$. The groups $H^k_c(F(X,n),k)$ depend only on the graded vector space $H^\bullet_c(X,k)$. If $k=\Q$, the same is true for the decomposition of $H^k_c(F(X,n),k)$ into irreducible representations. 
\end{thm}

\begin{rem}
	In fact Arabia did more: he studied certain generalized configuration spaces denoted $\Delta_{\leq \ell} X^n$ and $\Delta_\ell X^n$, where $\Delta_n X^n = F(X,n)$, and gave explicit formulas for their compact support cohomology in terms of the compact support cohomology of $X$ in the $i$-acyclic case. The spaces $\Delta_{\leq \ell} X^n$ and $\Delta_\ell X^n$ also fit into our framework: $\Delta_{\leq \ell} X^n$ can be written as $D(X,U)$, where $U \subset \Pi_n$ is the subposet consisting of partitions with at most $\ell$ blocks, and $\Delta_\ell X^n$ is in general a disjoint union of configuration spaces $F(X,\ell)$. 
\end{rem}

\begin{para}
	 We have already seen in Corollary \ref{cor-cdga} that $H^k_c(F(X,n),\Q)$ depends only on a cdga model for the compactly supported cochains $C^\bullet_c(X,\Q)$. One might therefore hope for the following to be true, which would re-prove and re-interpret Arabia's result: if a space $X$ is $i$-acyclic over $\Q$, then  the cdga $C^\bullet_c(X,\Q)$ is formal, and the cup product on $H^\bullet_c(X,\Q)$ vanishes. Equivalently, there is a  quasi-isomorphism between $C^\bullet_c(X,\Q)$ and $H^\bullet_c(X,\Q)$, where the cohomology is considered as a cdga with zero differential and zero multiplication. This turns out to be the case. Moreover, our method of proof will work with minor modifications to prove a result over an arbitrary coefficient ring. Let us state the main theorems of this section. \end{para}

\begin{thm}\label{i-acyclic-1}
	Let $X$ be a paracompact Hausdorff locally compact space which is $i$-acyclic over $\Q$. Then a cdga model for $C^\bullet_c(X,\Q)$ is given by $H^\bullet_c(X,\Q)$, considered as a cdga with identically zero differential and multiplication.
\end{thm}

\begin{rem}Theorem \ref{i-acyclic-1} has the following easy consequence. Let $X$ be a space whose one-point compactification $\bar X$ is nilpotent. If $X$ is $i$-acyclic then $\bar X$ has the based rational homotopy type of a wedge of spheres. Indeed, (a cdga model of) $C_c^\bullet(X,\Q)$ being formal is equivalent to $C^\bullet(\bar X,\Q)$ being formal as an augmented algebra, since $C_c^\bullet(X,\Q)$ is just the augmentation ideal. For example, we have seen that if $Y$ is arbitrary, then $Y \times \R$ is $i$-acyclic, and so the claim is that the one-point compactification of $Y \times \R$ should have the rational homotopy type of a wedge of spheres. But the one-point compactification of $Y \times \R$ is the suspension of the one-point compactification of $Y$, so this recovers the familiar fact that any suspension is a wedge of spheres rationally.
\end{rem}

\begin{cor}\label{corollary1}
	Let $X$ be a paracompact Hausdorff locally compact space which is $i$-acyclic over $\Q$. Let $U \subset \Pi_n$ be upwards closed. Then the cohomology groups $H^\bullet_c(F(X,U),\Q)$ and  $H^\bullet_c(D(X,U),\Q)$ depend only on the graded vector space $H^\bullet_c(X,\Q)$. In fact we have\[H^k_c(D(X,U),\Q) \cong \bigoplus_{T \in J_U}\bigoplus_{p+q=k} H^p_c(X^{\vert T \vert},\Q) \otimes \widetilde H^q(\nerveu {J_U^{\preceq T}},\Q)\]
	and 
	\[ H^k_c(F(X,U),\Q) \cong \bigoplus_{T \in J_{U_0}}\bigoplus_{p+q=k} H^p_c(X^{\vert T \vert},\Q) \otimes \widetilde H^q(\nerveud {J_{U_0}^{\preceq T}},\Q)\]
	These isomorphisms are equivariant in the sense that if $G \subseteq S_n$ is the subgroup preserving $U \subset \Pi_n$ (meaning that $G \cdot U \subseteq U$), then the isomorphisms are equivariant with respect to the action of $G$ on both sides. 
\end{cor}

\begin{thm}\label{i-acyclic-2}Let $X$ be a paracompact Hausdorff locally compact space which is $i$-acyclic over the ring $R$, and assume that $H^\bullet_c(X,R)$ is a projective $R$-module. Consider the tcdga given by $R\Gamma_c^\otimes(X,R)$. Its cohomology is the ``constant'' tcdga (\S \ref{constant-tcdga}) given by $H^\bullet_c(X,R)$ with identically zero multiplication and differential. Moreover, $R\Gamma_c^\otimes(X,R)$ is quasi-isomorphic to its cohomology in the category of commutative shuffle dg algebras. 
\end{thm}

\begin{cor}\label{i-acyclic-2-cor}Let $X$ be a paracompact Hausdorff locally compact space which is $i$-acyclic over the ring $R$, and assume that $H^\bullet_c(X,R)$ is a projective $R$-module. Let $U \subset \Pi_n$ be upwards closed. Then the cohomology groups $H^\bullet_c(F(X,U),R)$ and  $H^\bullet_c(D(X,U),R)$ depend only on the graded $R$-module $H^\bullet_c(X,R)$. In fact we have\[H^k_c(D(X,U),R) \cong \bigoplus_{T \in J_U}\bigoplus_{p+q=k} H^p_c(X^{\vert T \vert},\widetilde H^q(\nerveu {J_{U}^{\preceq T}},R))\]
	and 
	\[ H^k_c(F(X,U),R) \cong \bigoplus_{T \in J_{U_0}}\bigoplus_{p+q=k} H^p_c(X^{\vert T \vert},\widetilde H^q(\nerveud {J_{U_0}^{\preceq T}},R)). \]
	These isomorphisms are equivariant in a weak sense: if $G \subseteq S_n$ is the subgroup preserving $U \subset \Pi_n$, then there is a filtration on the left hand side for which the associated graded is isomorphic to the right hand side as a $G$-module.
\end{cor}

\begin{examplex}
	Suppose that $X = \mathbf R$. Then $F(X,U)$ is the complement of a hypergraph subspace arrangement in $\mathbf R^n$ (also known as a diagonal arrangement), and any hypergraph subspace arrangement arises for an appropriate choice of $U$. In this case $F(X,U)$ is an oriented manifold and its compact support cohomology equals its homology, so Corollary \ref{i-acyclic-2-cor} gives a formula for the integral homology of the complement of the arrangement in terms of the cohomology of the posets $J_U^{\prec T}$. The resulting formula is exactly the Goresky--MacPherson formula \cite{goreskymacphersonstratifiedmorsetheory}, in the special case of a hypergraph arrangement. What is perhaps surprising is that we obtain a formula of exactly the same ``shape'' for any $i$-acyclic space whatsoever; it is not at all clear from existing proofs of the Goresky--MacPherson formula that such a formula should exist for (say) $X=\R \times \Sigma$, where $\Sigma$ is a compact surface. Note also that the Goresky--MacPherson formula is not obviously equivariant: the equivariant Goresky--MacPherson formula (for $\Q$-coefficients) is a theorem of Sundaram and Welker \cite{sundaramwelker}. 
\end{examplex}

\begin{examplex}
	Another simple application of Corollary \ref{i-acyclic-2-cor} is that if the compact support cohomology of an $i$-acyclic space $X$ is torsion free, and the posets $J_{U_0}^{\preceq T}$ have torsion free cohomology (e.g.\ $U$ is Cohen--Macaulay), then all the spaces $F(X,U)$ have torsion free compact support cohomology. The property that $\widetilde H^\bullet(\nerveud{J_{U_0}^{\preceq T}},\mathbf Z)$ is torsion free is well studied in the subject of the topology of arrangements; it is equivalent to the corresponding complement of a subspace arrangement in $\mathbf R^n$ having torsion free homology.  
\end{examplex}

\subsection*{Recollections on $\mathsf A_\infty$- and $\mathsf C_\infty$-algebras}

\begin{para}
	In this subsection we very briefly recall some definitions from the theory of $\mathsf A_\infty$- and $\mathsf C_\infty$-algebras. The reader who is not familiar with these notions may consult \cite[Chapter 10]{lodayvallette} for a detailed account. 
\end{para}

\begin{para}
	We define an \emph{$\mathsf A_\infty$-algebra structure} on a graded $R$-module $A$ to be a degree $1$ square-zero coderivation of the tensor coalgebra $\mathsf{coAss}(A[1])$. A degree $1$ coderivation of $\mathsf{coAss}(A[1])$ determines and is uniquely determined by linear maps of degree $2-n$
	$$ \mu_n \colon A^{\otimes n} \to A$$
	for $n \geq 1$. That the coderivation squares to zero implies e.g.\ that $\mu_1$ is a differential making $A$ into a cochain complex and that $\mu_2$ is a multiplication satisfying the Leibniz rule. Hence there is an induced multiplication on $H^\bullet(A)$, which turns out to be associative. However, $\mu_2$ is not itself associative. Informally, the idea of an $\mathsf A_\infty$-algebra is that it is an algebra in which the multiplication $\mu_2$ satisfies the associativity laws up to coherent homotopy. For example, the next equation one obtains from the condition that the coderivation squares to zero is that the associator of $\mu_2$, considered as a map $A^{\otimes 3} \to A$, is the differential of the map $\mu_3$. When taking the cohomology one then gets associativity on the nose. In general one gets an infinite sequence of quadratic equations expressing relations between the maps $\mu_n$.
\end{para}

\begin{para}
	If $A$ is an $\mathsf A_\infty$-algebra, then we denote $\mathsf{coAss}(A[1])$, with its differential given by the $\mathsf A_\infty$-structure, by $BA$, and we call it the \emph{bar construction} on $A$. The classical {bar construction} on an associative dg algebra $A$ of Eilenberg--MacLane is a codifferential on the tensor coalgebra on the suspension of $A$, as above; in this way, every dg algebra may be considered as an $\mathsf A_\infty$-algebra in a natural manner. 
\end{para}
	
\begin{para}
	We define an \emph{$\mathsf A_\infty$-morphism} $A \to A'$ to be a morphism of dg coalgebras $f \colon BA \to BA'$. Such a morphism determines and is determined by a collection of linear maps of degree $1-n$
	$$ f_n \colon A^{\otimes n} \to A'.$$
	The condition that $f$ is a chain map then results in an infinite sequence of quadratic equations relating the maps $f_n$ and the $\mathsf A_\infty$-structures on $A$ and $A'$. We say that $f$ is an \emph{$\mathsf A_\infty$-quasi-isomorphism} if $f_1$ is a quasi-isomorphism. If $A=A'$ and $f_1$ is the identity map, then we say that $f$ is an \emph{$\mathsf A_\infty$-isotopy.}
\end{para}

\begin{para}
	We define similarly a $\mathsf C_\infty$-algebra structure on $A$ to be a degree $1$ square-zero coderivation of the cofree conilpotent Lie coalgebra $\mathsf{coLie}(A[1])$. Morphisms, quasi-isomorphisms, and isotopies of $\mathsf C_\infty$-algebras are defined in a completely analogous manner. Via the projection $\mathsf{coAss}(A[1]) \to \mathsf{coLie}(A[1])$ we may consider every $\mathsf C_\infty$-algebra to be an $\mathsf A_\infty$-algebra, and every $\mathsf C_\infty$-morphism to be an $\mathsf A_\infty$-morphism. Every commutative dg algebra may be considered as a $\mathsf C_\infty$-algebra. Informally, a $\mathsf C_\infty$-algebra is a dg algebra in which associativity only holds up to coherent higher homotopy, but whose multiplication is strictly commutative. Moreover, the higher homotopies are themselves commutative in a certain weak sense. 
\end{para}

\begin{para}\label{whatweuse}Assume now that the ground ring $R$ contains $\Q$. We will need only the following basic properties of the categories of $\mathsf A_\infty$- and $\mathsf C_\infty$-algebras:
	\begin{enumerate}
		\item If $A$ and $A'$ are dg algebras connected by an $\mathsf A_\infty$-quasi-isomorphism $A \to A'$, then there is a zig-zag of quasi-isomorphisms of dg algebras connecting $A$ and $A'$. 
		\item If $A$ and $A'$ are commutative dg algebras connected by a $\mathsf C_\infty$-quasi-isomorphism $A \to A'$, then there is a zig-zag of quasi-isomorphisms of commutative dg algebras connecting $A$ and $A'$. 
		\item If $A$ is a dg algebra, then there exists an $\mathsf A_\infty$-algebra structure on $H(A)$ and an $\mathsf A_\infty$-quasi-isomorphism $H(A) \to A$. The induced $\mathsf A_\infty$-structure on $H(A)$ is unique up to noncanonical $\mathsf A_\infty$-isotopy. 
		
		\item If $A$ is a commutative dg algebra, then there exists a $\mathsf C_\infty$-algebra structure on $H(A)$ and a $\mathsf C_\infty$-quasi-isomorphism $H(A) \to A$. The induced $\mathsf C_\infty$-structure on $V$ is unique up to noncanonical $\mathsf C_\infty$-isotopy. 
	\end{enumerate}
	For this, see \cite[Theorem 11.4.9]{lodayvallette} for the first two points and \cite[Theorem 10.3.1]{lodayvallette} for the third and fourth. The third and fourth property are theorems of Kadeishvili; see \cite{kadeishvili} for the $\mathsf A_\infty$-case.
\end{para}

\subsection*{Proof of Theorem \ref{i-acyclic-1}}

\begin{lem}\label{arabialemma0}Let $A$ be a differential graded algebra over $\Q$, and let $I \subset A$ be an ideal. Suppose that the induced map $H(I) \to H(A)$ vanishes. Then $I$ is formal as a dg algebra, and the multiplication on $H(I)$ is identically zero.  
\end{lem}

\begin{proof}We construct an $\mathsf A_\infty$-quasi-isomorphism from $H(I)$, considered as a dga with vanishing multiplication and differential, to $I$. This is enough to conclude the result by \S\ref{whatweuse}(1). Unwinding the definitions, this means that we have to find maps
	$$ f_n \colon H(I)^{\otimes n} \to I, \qquad \qquad n \geq 1$$
	of degree $1-n$, such that
	$$ d \circ  f_n = \sum_{i+j=n}(-1)^i f_i \cdot f_j. $$
	Let $f \colon H(I) \to I$ be a map taking every class to a representing cocycle. Let $g \colon H(I) \to A$ be a map such that $d \circ g = -f$. Such maps exist because we work over a field, and since every cocycle in $I$ is a coboundary in $A$. 
	Now define for all $n \geq 1$,  
	$$ f_n(x) = f(x) \cdot \underbrace{g(x) \cdot \ldots \cdot g(x)}_{(n-1) \text{ times}}. $$
	Since $I$ is an ideal in $A$, this product is well defined as an element of $I$. One checks using the Leibniz rule that the collection $\{f_n\}$ defines an $\mathsf A_\infty$-morphism;  it is clearly a quasi-isomorphism. 
\end{proof}

\begin{lem}\label{arabialemma}Let $A$ be a cdga over $\Q$, and let $I \subset A$ be an ideal. Suppose that the induced map $H(I) \to H(A)$ vanishes. Then $I$ is formal as a cdga, and the multiplication on $H(I)$ is identically zero.  
\end{lem}

\begin{proof}By \S\ref{whatweuse}(4) we may give $H(I)$ some $\mathsf C_\infty$-algebra structure making it $\mathsf C_\infty$-quasi-isomorphic to $I$. On the other hand we may give $H(I)$ an $\mathsf A_\infty$-algebra structure by setting all operations identically zero, and by the previous lemma this structure is $\mathsf A_\infty$-quasi-isomorphic to $I$. By the uniqueness part of \S\ref{whatweuse}(3) this means that the transferred $\mathsf C_\infty$-structure on $H(I)$ must be $\mathsf A_\infty$-isotopic to the degenerate $\mathsf A_\infty$-structure with all operations zero. While it is hard in general to determine whether two $\mathsf A_\infty$-structures are $\mathsf A_\infty$-isotopic, the case when one of them is identically zero is easy: the only thing isotopic to the identically zero  $\mathsf A_\infty$-structure is the identically zero  $\mathsf A_\infty$-structure.
\end{proof}

\begin{rem}One can also deduce Lemma \ref{arabialemma} from Lemma \ref{arabialemma0} by appealing to \cite[Theorem 1.3]{saleh}.
\end{rem}

\begin{thm}[Theorem \ref{i-acyclic-1} restated]
	Let $X$ be a paracompact Hausdorff locally compact space which is $i$-acyclic over $\Q$. Then a cdga model for $C^\bullet_c(X,\Q)$ is given by $H^\bullet_c(X,\Q)$, considered as a cdga with identically zero differential and multiplication.
\end{thm}

\begin{proof}We may take for our cdga model for $C^\bullet_c(X,\Q)$ the algebra $R\Gamma_c^{TW}(X,\Q)$. It is an ideal inside the cdga $R\Gamma^{TW}(X,\Q)$ and the induced map $H^\bullet_c(X,\Q) \to H^\bullet(X,\Q)$ is zero by assumption. The result follows from Lemma \ref{arabialemma}.\end{proof}

\subsection*{Proof of Theorem \ref{i-acyclic-2}}

\begin{para}
	The proof of Theorem \ref{i-acyclic-2} is nearly identical to the proof of Theorem \ref{i-acyclic-1}. We will however need to work with $\mathsf C_\infty$- and $\mathsf A_\infty$-algebras in more general categories than $\mathsf{Ch}_R$, which serves to add an additional layer of abstraction; moreover, since we do not work over $\Q$ we should argue that the statements of \S\ref{whatweuse} remain valid. 
	\end{para} 
	
\begin{para}
	We will need to consider $\mathsf A_\infty$-algebras and $\mathsf C_\infty$-algebras in the category of functors $\mathsf{Mod}_R^{\mathsf{Ord}_+}$, with the ``shuffle'' monoidal structure introduced in \S \ref{diagram}: 
	$$ (A \otimes B)(S) = \bigoplus_{S = T \sqcup T'} A(T) \otimes B(T').$$
	That is, we repeat verbatim the definitions of $\mathsf A_\infty$-algebra and $\mathsf C_\infty$-algebra, but the tensor coalgebra and cofree conilpotent Lie coalgebra is now computed using the above tensor product. An associative dg algebra in this category of functors is nothing but a a shuffle dg algebra, and a commutative dg algebra is nothing but a commutative shuffle dg algebra, as defined in \S\ref{defn-shuffle}. 	 We shall refer to $\mathsf A_\infty$-algebras in this category as \emph{shuffle $\mathsf A_\infty$-algebras}. They form a natural enlargement of the category of shuffle dg algebras. Similarly we may consider shuffle $\mathsf C_\infty$-algebras as an enlargement of the category of commutative shuffle dg algebras. 
\end{para}

\begin{prop}\label{shuffle-qiso}
	Let $A$ and $A'$ be shuffle dg algebras. If there exists an $A_\infty$-quasi-isomorphism $f \colon A \to A'$, then there exists a zig-zag of quasi-isomorphisms of shuffle dg algebras connecting $A$ and $A'$. Similarly for commutative shuffle dg algebras and shuffle $\mathsf C_\infty$-algebras. 
\end{prop}

\begin{para}
	Before giving the proof of Proposition \ref{shuffle-qiso}, let us recall how one proves this for usual $\mathsf A_\infty$-algebras and $\mathsf C_\infty$-algebras over $\Q$. The idea is to show that each morphism in the zig-zag
	$$ A \leftarrow \Omega BA \stackrel {\Omega f}\longrightarrow  \Omega BA' \to A'$$
	is a quasi-isomorphism, where the first and last arrows are given by the counit of the bar-cobar adjunction \cite[Section 11.3]{lodayvallette} and the second arrow is given by applying the functor cobar to the map $BA \to BA'$ corresponding to the $\mathsf A_\infty$-morphism $f$. Disregarding the differential we have
	$$ \Omega BA \cong \mathsf{Ass}(s\mathsf{coAss}(s^{-1}A)) \cong \mathsf{Ass}(\Lambda^{-1}\mathsf{coAss}(A)) \cong 	\bigoplus_{k\geq 1} (\mathsf{Ass} \circ \Lambda^{-1}\mathsf{coAss})(k) \otimes_{\mathbb S_k} A^{\otimes k},$$
	where $\mathsf{Ass}$ denotes the associative operad, $\mathsf{coAss}$ the coassociative cooperad, $\Lambda^{-1}$ the operadic desuspension. (That is, $\Lambda^{-1} \mathsf{coAss}$ is the Koszul dual cooperad of $\mathsf{Ass}$.) See \cite[Chapter 11]{lodayvallette}. We can define an increasing filtration $L$ on $\Omega BA$ by 
	$$ L_m \Omega BA = \bigoplus_{k=1}^m (\mathsf{Ass} \circ \Lambda^{-1}\mathsf{coAss})(k) \otimes_{\mathbb S_k} A^{\otimes k}.$$
	We also give $A$ the trivial filtration with $L_0A = \{0\}$ and $L_1A = A$, and we define similarly filtrations on $\Omega BA'$ and $A'$. We will argue that all maps in the zig-zag $ A \leftarrow \Omega BA \to  \Omega BA' \to A'$ are filtered quasi-isomorphisms, which is enough to conclude the result since a filtered quasi-isomorphism of complexes with bounded below exhaustive filtrations is a quasi-isomorphism. 
	
	Now the induced differential on $\operatorname{Gr}^L_m \Omega BA \cong (\mathsf{Ass} \circ \Lambda^{-1}\mathsf{coAss})(m) \otimes_{\mathbb S_m} A^{\otimes m} $ is given only by the Koszul differential on $(\mathsf{Ass} \circ \Lambda^{-1}\mathsf{coAss})(m)$. But this Koszul complex is acyclic for $m>1$; indeed, this is the definition of Koszul duality. So certainly each morphism in the zig-zag is a quasi-isomorphism on $\operatorname{Gr}^L_m$ for $m>1$. For $m=1$ we have an identification $\operatorname{Gr}^L_1\Omega BA = A$, and applying $\operatorname{Gr}^L_1$ to the zig-zag gives simply
	$$ A \stackrel = \longleftarrow A \stackrel {f_1} \longrightarrow A' \stackrel =\longrightarrow A',$$
	and the result follows. The proof for $\mathsf C_\infty$-algebras over $\Q$ works the same, except we replace $\mathsf{Ass} \circ \Lambda^{-1}\mathsf{coAss}$ with the Koszul complex $\mathsf{Com} \circ \Lambda^{-1}\mathsf{coLie}$. The only place where we use that the ground ring is $\Q$ is when we deduce that $(\mathsf{Ass} \circ \Lambda^{-1}\mathsf{coAss})(m) \otimes_{\mathbb S_m} A^{\otimes m} $ is acyclic from the fact that $(\mathsf{Ass} \circ \Lambda^{-1}\mathsf{coAss})(m)$ is acyclic --- over a general ground ring, the functor $- \otimes_{\mathbb S_m} A^{\otimes m}$ has no reason to be exact in general. (In fact for $\mathsf A_\infty$-algebras the result is valid over $\Z$, since $(\mathsf{Ass} \circ \Lambda^{-1}\mathsf{coAss})(m)$ turns out to be a complex of free $\Z[\mathbb S_m]$-modules. But for $\mathsf C_\infty$-algebras we do need to work over $\Q$ in general.)
\end{para}

\begin{proof}[Proof of Proposition \ref{shuffle-qiso}] As in the previous paragraph we write down the zig-zag
	$$ A \longleftarrow \Omega BA \longrightarrow \Omega BA' \longrightarrow A',$$
	and we define the same filtration on each of the terms. We get in the same way an isomorphism
	$$\operatorname{Gr}^L_m \Omega BA \cong (\mathsf{Ass} \circ \Lambda^{-1}\mathsf{coAss})(m) \otimes_{\mathbb S_m} A^{\otimes m}, $$
	where we now use that the category $\mathsf{Ch}_R^{\mathsf{Ord}_+}$ is tensored over $\mathsf{Ch}_R$. Now we need to use an observation that we have already made in Theorem \ref{ps}: that $A^{\otimes m}$ is a free $R[\mathbb S_m]$-module for any $A \in \operatorname{ob} \mathsf{Ch}_R^{\mathsf{Ord}_+}$, as follows from the calculation
	$$ A^{\otimes m}(k) =\!\! \bigoplus_{\{1,\ldots,k\} = S_1 \sqcup \ldots \sqcup S_m}\!\! A(S_1) \otimes \ldots \otimes A(S_m) = R[\mathbb S_m] \otimes \!\!\bigoplus_{\substack{\{1,\ldots,k\} = S_1 \sqcup \ldots \sqcup S_m\\ \min(S_1) < \ldots < \min(S_m)}}\!\! A(S_1) \otimes \ldots \otimes A(S_m).$$	
	In particular, since $(\mathsf{Ass} \circ \Lambda^{-1}\mathsf{coAss})(m)$ is acyclic for $m>1$, we deduce that $\operatorname{Gr}^L_m \Omega BA$ is acyclic for $m>1$. Thus the proof goes through as in the preceding paragraph without changes. For $\mathsf C_\infty$-algebras we run the same argument except we replace $\mathsf{Ass} \circ \Lambda^{-1}\mathsf{coAss}$ with the Koszul complex $\mathsf{Com} \circ \Lambda^{-1}\mathsf{coLie}$.
\end{proof}

\begin{prop}
	Let $A$ be a shuffle  dg algebra such that $H(A)$ is a projective $R$-module in arity $n$ for all $n$.  
	There exists a shuffle $\mathsf A_\infty$-algebra structure on $H(A)$ and an $\mathsf A_\infty$-quasi-isomorphism $H(A) \to A$. The induced shuffle $\mathsf A_\infty$-structure on $H(A)$ is unique up to noncanonical $\mathsf A_\infty$-isotopy. Similarly the cohomology of a commutative shuffle dg algebra can be given a shuffle $\mathsf C_\infty$-algebra structure with a $\mathsf C_\infty$-quasi-isomorphism $H(A) \to A$.
\end{prop}

\begin{proof}
	The projectivity assumption implies the existence of a quasi-isomorphism $f \colon H(A) \to A$ in $\mathsf{Ch}_R^{\mathsf{Ord}_+}$. According to \cite[Theorem 2]{htt}, the conclusion will follow once we have verified that $f$ induces a quasi-isomorphism
	$$ \underline{\mathrm{Hom}}_{\mathsf{Ch}_R^{\mathsf{Ord}_+}}(\mathsf{coAss}(H(A)),H(A)) \longrightarrow \underline{\mathrm{Hom}}_{\mathsf{Ch}_R^{\mathsf{Ord}_+}}(\mathsf{coAss}(H(A)),A),$$
	where $\underline{\mathrm{Hom}}$ denotes the internal Hom; for $\mathsf C_\infty$-algebras we should instead consider $$ \underline{\mathrm{Hom}}_{\mathsf{Ch}_R^{\mathsf{Ord}_+}}(\mathsf{coLie}(H(A)),H(A)) \longrightarrow \underline{\mathrm{Hom}}_{\mathsf{Ch}_R^{\mathsf{Ord}_+}}(\mathsf{coLie}(H(A)),A).$$
	In both cases, the result is immediate from the fact that $H(A)$ is projective, and then so are $\mathsf{coAss}(H(A))$ and $\mathsf{coLie}(H(A))$, so $\underline{\mathrm{Hom}}_{\mathsf{Ch}_R^{\mathsf{Ord}_+}}(\mathsf{coAss}(H(A)),-)$ and $\underline{\mathrm{Hom}}_{\mathsf{Ch}_R^{\mathsf{Ord}_+}}(\mathsf{coLie}(H(A)),-)$ are both exact functors.  The proof in \cite{htt} is formulated in $\mathsf{Ch}_R$ but works with no modifications in $\mathsf{Ch}_R^{\mathsf{Ord}_+}$. 
\end{proof}

\begin{rem}There are very many ways of proving that one can transfer $\mathsf A_\infty$- or $\mathsf C_\infty$-algebra structure to cohomology, e.g.\ using sums over trees \cite[Section 10.3]{lodayvallette}, homological perturbation theory \cite{berglundhomologicalperturbation}, or abstract model-categorical arguments \cite{bergermoerdijk}. In the preceding proof we refer to the paper \cite{htt} since the other approaches mentioned above all require stronger assumptions on the ground ring or on the operads or complexes involved; we do not want to assume for example that $H(A) \to A$ is a homotopy equivalence. For $\mathsf A_\infty$-algebras the proof in \cite{htt} is exactly the same as the one of Kadeishvili \cite{kadeishvili}. According to \cite{kadeishvili-Coo}, Kadeishvili proved the analogous transfer result for $\mathsf C_\infty$-algebras over a general ground ring as well in the 1980's, but these references are not easily available, and Kadeishvili's definition of a $\mathsf C_\infty$-algebra is slightly different from the one used here; the two notions are equivalent only over a ground ring containing $\Q$. 
\end{rem}

\begin{lem}
	Let $A$ be a commutative shuffle dg algebra, $I \subset A$ an ideal. Suppose that the induced map $H(I) \to H(A)$ vanishes, and that $H(I)$ is projective --- that is, $H(I)(n)$ is degreewise a projective $R$-module, for every $n \geq 1$. Then $I$ is formal as a commutative shuffle dg algebra, and the multiplication on $H(I)$ is identically zero.
\end{lem}

\begin{proof}
	Repeat the proof of Lemma \ref{arabialemma}, with the tacit understanding that all $\mathsf C_\infty$-algebras are now shuffle $\mathsf C_\infty$-algebras, all $\mathsf A_\infty$-algebras are now shuffle $\mathsf A_\infty$-algebras, etc., and where a quasi-isomorphism $H(I) \to I$ in $\mathsf{Ch}_R^{\mathsf{Ord}_+}$ now exists since we assumed $H(I)(n)$ to be a projective $R$-module for all $n$. 
\end{proof}

\begin{rem}
	The preceding proof shows why we need to work with commutative shuffle algebras instead of twisted commutative algebras when we are not over a field of characteristic zero: in order to construct a quasi-isomorphism $H(I) \to I$ in $\mathsf{Ch}_R^{\mathsf{Ord}_+}$ we only need $H(I)(n)$ to be a projective $R$-module, but if we wanted a quasi-isomorphism $H(I) \to I$ in in $\mathsf{Ch}_R^{\mathsf{Fin}_+}$
 we would need $H(I)(n)$ to be a projective $R[\mathbb S_n]$-module for all $n$, which will practically never be true. \end{rem}

\begin{thm}[Theorem \ref{i-acyclic-2} restated] Let $X$ be a paracompact Hausdorff locally compact space which is $i$-acyclic over the ring $R$, such that $H^\bullet_c(X,R)$ is degreewise a projective $R$-module. Then $R\Gamma_c^\otimes(X,R)$ is formal as a commutative shuffle dg algebra, and the multiplication on its cohomology $H^\bullet_c(X,R)$ is identically zero. \end{thm}

\begin{proof}
	Apply the preceding lemma to the ideal $R\Gamma_c^{\otimes}(X,R)$ inside the commutative shuffle dg algebra $R\Gamma^\otimes(X,R)$.  
\end{proof}

\begin{rem}
	If $H^\bullet_c(X,R)$ is not a projective $R$-module there is still something one can say. By a modification of the above arguments one can show that if $X$ is ``derived $i$-acyclic'', meaning that the map $C_c^\bullet(X,R) \to C^\bullet(X,R)$ is zero in the derived category of dg $R$-modules, then $R\Gamma_c^\otimes(X,R)$ is quasi-isomorphic as a commutative dg shuffle algebra to a projective resolution of $H^\bullet_c(X,R)$, considered as a commutative dg shuffle algebra with identically zero multiplication. When the cohomology is projective as an $R$-module, $i$-acyclicity is equivalent to derived $i$-acyclicity. We omit the details.
\end{rem}

\section{Example application: $k$-equals configuration spaces}
\label{section:k-equals}

\begin{para}
	In this section, we will illustrate how our results may be applied. Let $X$ be an $i$-acyclic space, $k \geq 2$ an integer, and let $F_{< k}(X,n)$ denote the \emph{$k$-equals configuration space} --- the space parametrizing $n$ ordered points on $X$ such that no subset of $k$ points coincide. In this section we write down explicit formulae for the cohomology groups $H^i_c(F_{<  k}(X,n),\Q)$ and their decompositions into irreducible $\mathbb S_n$-representations, in terms of the compactly supported Poincar\'e polynomial of $X$. We assume throughout that $X$ is paracompact locally compact Hausdorff and that all of the cohomology groups $H^i_c(X,\Q)$ are finite dimensional $\Q$-vector spaces. 
\end{para}
	
	 \begin{para}Our calculation is formulated in terms of the algebra of symmetric functions, and we will use the correspondence between symmetric functions and symmetric sequences. Let us briefly recall the relevant terminology, although part of it has already been used (in greater generality) in previous sections of this paper. 
\end{para}

\subsection*{Symmetric sequences and a reformulation of the calculation}

\begin{defn}To avoid repeating hypotheses, we will for the remainder of this section use the term \emph{symmetric sequence} to mean a sequence $\mathbf A = \{\mathbf A(n)\}_{n =0}^\infty$ of representations of the symmetric groups $\mathbb S_n$ in the category of graded, degreewise finite dimensional $\mathbf Q$-vector spaces. We have already used the notion of symmetric sequence previously in the paper, in greater generality: we considered  the tensor product of symmetric sequences in \S \ref{monoidalstructure}; moreover, we have freely used the language of operads, and an operad is nothing but a monoid in the category of symmetric sequences with respect to the monoidal structure given by composition product.
\end{defn}

\begin{defn}
	Let $\mathbf A$ and $\mathbf B$ be symmetric sequences. Their \emph{tensor product} is defined by $$(\mathbf A\otimes \mathbf B)(n) = \bigoplus_{n=k_1+k_2} \mathrm{Ind}_{\mathbb S_{k_1} \times \mathbb S_{k_2}}^{\mathbb S_n} \mathbf A(k_1) \otimes \mathbf B(k_2).$$
 If $\mathbf B(0)=0$, then their \emph{composition product} is the symmetric sequence defined by
	$$ (\mathbf A \circ \mathbf B)(n) = \bigoplus_{k=0}^\infty \mathbf A(k) \otimes_{\mathbb S_k} \mathbf B^{\otimes k}(n), $$
	where $\mathbf B^{\otimes k}$ denotes the $k$-fold tensor product of $\mathbf B$ with itself. (The definition makes sense also when $\mathbf B(0)\neq 0$, but $\mathbf A \circ \mathbf B$ might no longer be degreewise finite dimensional as a vector space in each arity.)
\end{defn}

\begin{para}\label{some-s-modules}	Let $\Pi_{(k,1^{n-k})} \subset \Pi_n$ be the subposet of partitions all of whose blocks are either singletons or have size $\geq k$. Let $\mathbf P_k$ denote the symmetric sequence with
	$$ \mathbf P_k(n) =  \widetilde H^\bullet(\nerveud {\Pi_{(k,1^{n-k})}},\Q)$$
	for $n > 0$.  Let $\mathbf E$ denote the symmetric sequence with $\mathbf E(n) = \Q$, the trivial representation concentrated in degree $0$, for all $n \geq 0$. Finally we consider $H^\bullet_c(X,\Q)$ as a symmetric sequence concentrated in arity $0$. 
\end{para}

\begin{rem}\label{rem1}
	Note that the composition product $H^\bullet_c(X,\Q) \circ \mathbf P_k$ is just the aritywise tensor product: \[(H^\bullet_c(X,\Q) \circ \mathbf P_k)(n) = H^\bullet_c(X,\Q) \otimes \mathbf P_k(n).\]
\end{rem}
\begin{rem}\label{rem2}
	The composition product $\mathbf E \circ \mathbf A$ can be understood as an ``exponential'' of $\mathbf A$, and can be explicitly written as 
	$$(\mathbf E \circ \mathbf A)(n) = \bigoplus_{\substack{T \in \Pi_n\\ \vert T_i \vert = t_i, \,\, \, i=1,\ldots,\ell}} \bigotimes_{i=1}^\ell \mathbf A(t_i)$$
	where the summation indicates that we sum over all partitions $T$, and that that the partition $T$ has $\ell$ blocks of size $t_1,\ldots,t_\ell$.   
\end{rem}

\begin{prop}\label{plethysm-1}Let $X$ be $i$-acyclic over $\Q$. There is an isomorphism of $\mathbb S_n$-representations
	$$ (\mathbf E \circ H^\bullet_c(X,\Q) \circ \mathbf P_k )(n) = H^\bullet_c(F_{<  k}(X,n),\Q).$$
\end{prop}
\begin{proof}
	Let $U \subset \Pi_n$ be the upwards closed set consisting of all partitions containing a block of size at least $k$, so that $F_{<  k}(X,n) = F(X,U)$. As in \S\ref{bjorner} we denote by $J_{U_0}$ the set of joins of atoms in $U$, including the empty join $\hat 0$. Note that
	$J_{U_0}  = \Pi_{(k,1^{n-k})}.$ If $T \in J_{U_0}$ has blocks of size $t_1,\ldots,t_\ell$, then the poset $J_{U_0}^{\preceq T}$ decomposes as a cartesian product
	\[J_{U_0}^{\preceq T} \cong \prod_{i=1}^\ell \Pi_{(k,1^{t_i-k})},\]
	and hence Lemma \ref{smash} says that there is a homeomorphism
	\[\nerveud {J_{U_0}^{\preceq T}} \cong  \nerveud{\Pi_{(k,1^{t_1-k})}}  \wedge\nerveud{\Pi_{(k,1^{t_2-k})}} \wedge \ldots \wedge \nerveud{\Pi_{(k,1^{t_\ell-k})}}. \]
	But according to Corollary \ref{corollary1} there is an $\mathbb S_n$-equivariant isomorphism
	\begin{align*}
	H^\bullet_c(F_{\leq k}(X,n),\Q) & \cong \bigoplus_{T \in J_{U_0}}  \widetilde H^\bullet(\nerveud {J_{U_0}^{\preceq T}},\Q) \otimes H^\bullet_c(X^{\vert T\vert},\Q)\end{align*}
	and the above expression for $\nerveud {J_{U_0}^{\preceq T}}$ implies that the right hand side can be rewritten as
	\begin{align*}
	\bigoplus_{\substack{T \in J_{U_0} \\ \vert T_i \vert = t_i,\,\, i=1,\ldots,\ell}} \bigotimes_{i=1}^\ell \widetilde H^\bullet(\nerveud{\Pi_{(k,1^{t_i-k})}},\Q) \otimes H^\bullet_c(X,\Q)
	\end{align*}
	which is now nothing but the arity $n$ component of $\mathbf E \circ H^\bullet_c(X,\Q) \circ \mathbf P_k$, where we use Remark \ref{rem1} and \ref{rem2}.
\end{proof}

\subsection*{The formalism of symmetric functions}

\begin{para}
	Let $\Lambda$ denote the ring of symmetric functions over $\Q$. It is graded: $\Lambda = \bigoplus_{n =0}^\infty \Lambda_n$. We let $\widehat \Lambda$ denote the completion of $\Lambda$ with respect to the filtration by degree. Explicitly, $\widehat \Lambda = \prod_{n=0}^\infty \Lambda_n$. 	
\end{para}	

\begin{para}
	For each $n$ there is an isomorphism $\Lambda_n \cong R(\mathbb S_n)$ between $\Lambda_n$ and the Grothendieck group of representations of $\mathbb S_n$. Here it is crucial that we work over a field of characteristic zero.  If $M$ is a representation of $\mathbb S_n$, then we denote by $\operatorname{ch}_n M$ the corresponding symmetric function. 
\end{para}

\begin{para}
	Let $\EuScript R$ denote the ring
	$$ \EuScript R = \prod_{n=0}^\infty \Lambda_n \otimes \Q[t,t^{-1}],$$
	considered as a subring of $\widehat \Lambda[\![t,t^{-1}]\!]$. If $\mathbf M$ is a symmetric sequence, then we denote
	$$ \operatorname{ch} \mathbf M = \sum_{n=0}^\infty \sum_{i\in \Z} (-t)^i \operatorname{ch}_n \mathbf M(n)^i \in \EuScript R,$$
	where $\mathbf M(n)^i$ denotes the component of $\mathbf M(n)$ in cohomological degree $i$. Two symmetric sequences  $\mathbf M$ and $\mathbf M'$ are isomorphic if and only if $\operatorname{ch}\mathbf M = \operatorname{ch}\mathbf M'$. This is a consequence of the fact that the category of symmetric sequences is semisimple. The assignment $\mathbf M \mapsto \operatorname{ch}\mathbf M$ induces an isomorphism between $K_0(\mathsf{Mod}_{\mathbb S})$ (the Grothendieck group of symmetric sequences) and $\EuScript R$. This is in fact an isomorphism of \emph{commutative rings}, where $K_0(\mathsf{Mod}_{\mathbb S})$ carries the ring structure induced from the tensor product of symmetric sequences. 
\end{para}

\begin{para}
We will consider two particular bases for the ring $\Lambda$. Recall that irreducible representations of $\mathbb S_n$ are parametrized in a standard way by partitions $\lambda \vdash n$; we let $V_\lambda$ denote the irreducible representation (Specht module) corresponding to $\lambda$. The elements $\operatorname{ch}_n V_\lambda \in \Lambda_n$ are called \emph{Schur polynomials} and will be denoted $s_\lambda$. For example, $s_n$ corresponds to the trivial representation of $\mathbb S_n$, and $s_{1^n}$ the sign representation. Schur polynomials form a basis for $\Lambda$ as a vector space. We will also use the \emph{power sums} $p_n$, $n \geq 1$, which are uniquely determined by the equality of formal power series
$$ \sum_{n \geq 0} s_n t^n = \exp \sum_{n \geq 1} p_n t^n/n. $$
The power sums freely generate $\Lambda$ as a $\Q$-algebra.
\end{para}

\begin{para}
	The ring $\Lambda$ carries an operation called \emph{plethysm}, denoted $f \circ g$. Plethysm is uniquely determined by the following properties:
	\begin{itemize}
		\item For all $g \in \Lambda$, the map $f \mapsto f \circ g$ is a ring homomorphism.
		\item For all $n \geq 1$, the map $g \mapsto p_n \circ g$ is a ring homomorphism.
		\item $p_n \circ p_m = p_{nm}$. 
	\end{itemize}
	We extend plethysm to an operation on $\Lambda[t,t^{-1}]$ by imposing the additional rules:
	\begin{itemize}
		\item Plethysm is $\Q[t,t^{-1}]$-linear in the first variable.
		\item $p_n \circ t = t^{n}$.
	\end{itemize}
	Plethysm does not extend to a well defined operation on $\EuScript R$, but if $g$ has no constant term then we may define $f \circ g$ by the above rules. 
\end{para}

\begin{para}
	The isomorphism $K_0(\mathsf{Mod}_{\mathbb S}) \to \EuScript R$ carries the composition product of symmetric sequences to the plethysm of symmetric functions. Note that we defined the composition product $\mathbf M \circ \mathbf N$ of symmetric sequences only when $\mathbf N(0)=0$, and the plethysm $f \circ g$ only when $g$ has no constant term, so the claim should be understood as saying that the isomorphism is compatible with these partially defined functions.   
\end{para}

\subsection*{The calculation for $k$-equals configuration spaces}

\begin{para}Consider the following three elements of $\EuScript R$. 
\begin{align*}
\mathsf S_k & = -\sum_{n \geq k} (-t)^{n-k+2}s_{n-k+1,1^{k-1}}, \\
\mathsf L & =  \sum_{n \geq 1} \frac {1} n \sum_{d \vert n } (-1)^{n/d-1}\mu(d) p_d^{n/d},  \\
\mathsf E & = \sum_{n \geq 0} s_{n}. 
\end{align*} 
Here $\mu$ denotes the usual number-theoretic M\"obius function. The following is a theorem of Sundaram and Wachs \cite[Theorem 3.5]{sundaramwachs}, slightly reformulated. Their result was previously applied to the computations of the $\mathbb S_n$-equivariant rational cohomologies of the spaces $F_{\leq k}(\mathbf R^d,n)$ by Sundaram and Welker \cite{sundaramwelker}. The non-equivariant version of the calculation was done previously by Bj\"orner and Welker \cite{bjornerwelker}, where they also proved that the cohomology is torsion free.

\end{para}

\begin{thm}[Sundaram--Wachs]\label{sw}
		 Let $\mathbf P_k$ denote the symmetric sequence defined in \S \ref{some-s-modules}. There is an equality
 $$s_1 + t^{-1}\mathsf L \circ \mathsf S_k = \operatorname{ch} \mathbf P_k$$
 in $\EuScript R$.
\end{thm}

\begin{para}
	Let $\mathbf F_{<  k}(X)$ denote the symmetric sequence given by
	$$ \mathbf F_{<  k}(X)(n) = H^\bullet_c(F_{<  k}(X,n),\Q).$$
	By combining Proposition \ref{plethysm-1} with Theorem \ref{sw}, we can give an explicit formula for $\operatorname{ch} \mathbf F_{\leq k}(X)$ when $X$ is $i$-acyclic. In other words, we obtain a complete calculation of all cohomology groups $H^i_c(F_{<  k}(X,n),\Q)$ with their decompositions into irreducible representations of $\mathbb S_n$. 
\end{para}

\begin{cor}
	Let $X$ be an $i$-acyclic space over $\mathbf Q$, and let $h(t) = \sum_i (-t)^i \dim_\Q H^i_c(X,\Q)$. Then there is an equality
	$$ \mathsf E \circ (h(t) \cdot (s_1 + t^{-1} \mathsf L \circ \mathsf S_k)) = \operatorname{ch} \mathbf F_{<  k}(X) $$
	in $\EuScript R$.
\end{cor}

\begin{proof}
	Apply the operator $\operatorname{ch}$ to Proposition \ref{plethysm-1}, and use that composition product of symmetric sequences is mapped to plethysm of symmetric functions. Then plug in the calculation of $\operatorname{ch} \mathbf P_k$ given by Theorem \ref{sw}, and the obvious equalities $\operatorname{ch}\mathbf E = \mathsf E$ and $\operatorname{ch} H^\bullet_c(X,\Q) = h(t)$. 
\end{proof}

\begin{para}
	
When $k=2$, i.e.\ in the case of the usual configuration spaces of points, the above formula can be simplified, since there is a simpler formula for $\operatorname{ch} \mathbf P_2$. Namely, one has
$$ s_1 + t^{-1}\mathsf L \circ \mathsf S_2 = \sum_{n \geq 1} \frac {t^{n-1}} n \sum_{d \vert n } (-1)^{n/d-1}\mu(d) p_d^{n/d}.$$
\end{para}

\begin{rem}
	If $X$ is not assumed to be $i$-acyclic, then the formula $\mathsf E \circ (h(t) \cdot (s_1 + t^{-1} \mathsf L \circ \mathsf S_k))$ still computes the characteristic of the $E_1$ page of the spectral sequence calculating $H^\bullet_c(F_{<  k}(X,n),\Q)$. If we suppose that $H^\bullet_c(X,\Q)$ is finite dimensional (not just degreewise finite dimensional), then we may take the Euler characteristic of both sides (i.e.\ set $t=1$) to obtain a generating function for the $\mathbb S_n$-equivariant compactly supported Euler characteristic of $F_{<  k}(X,n)$:
	$$ \mathsf E \circ (\chi_c(X) \cdot (s_1 + \mathsf L \circ \sum_{n \geq k} (-1)^{n-k+1}s_{n-k+1,1^{k-1}}) = \sum_{n\geq 0} \chi^{\mathbb S_n}_c(F_{<  k}(X,n)),$$
	where this is now an equality in $\widehat \Lambda$.
	If $X$ is a complex algebraic variety we may replace $\chi_c(X)$ with the Hodge--Deligne polynomial of $X$, or even better the class $\sum_i (-1)^i [H^i_c(X,\Q)]$ in the Grothendieck ring $K_0(\mathsf{MHS}_\Q)$ of rational mixed Hodge structures. In that case the above formula may be considered as an equality in the completion of $\Lambda \otimes K_0(\mathsf{MHS}_\Q)$, valid for any complex algebraic variety. When $k=2$ these generating series for Hodge--Deligne polynomials were previously obtained by Getzler \cite{getzler99}.
\end{rem}

\section{Chevalley--Eilenberg homology of twisted Lie algebras}

\begin{para} The goal of this section is to explain that in the case of the usual configuration space $F(X,n)$, our functor $\CF$ can be reinterpreted in terms of operadic cohomology of twisted Lie algebras (left modules over the Lie operad). 
\end{para}

\begin{para}Let $A$ be a commutative algebra and $\mathfrak g$ a Lie algebra. Then the tensor product $A \otimes \mathfrak g$ is again a Lie algebra. If $\mathfrak g$ is instead a twisted dg Lie algebra (a left module over the Lie operad) and $A$ is a twisted commutative dg algebra, then their so-called Hadamard tensor product, defined by
	$$ (A \otimes_H \mathfrak g)(n) = A(n) \otimes \mathfrak g(n),$$
	is again a twisted Lie algebra. Equivalently, it can be considered as a Lie algebra in the symmetric monoidal category of symmetric sequences. As such, we can consider its Chevalley--Eilenberg (co)homology, which will itself be a symmetric sequence (equivalently, a sequence of representations of $\mathbb S_n$). We will find it more natural to work with Lie algebra \emph{homology}, so in this section we will change to homological grading via usual the convention $C^i = C_{-i}$ for $i \in \Z$, where $C^\bullet$ is any cochain complex. We denote the complex of Chevalley--Eilenberg chains by $C^{\mathrm{CE}}_\bullet(-)$.
\end{para}

\begin{para}
	We will also need the \emph{suspension} operation on twisted commutative algebras. If $A$ is a twisted commutative dg algebra, then we define
	$$ \mathsf SA (n) = A(n)[-n] \otimes \mathrm{sgn}_n.$$
	Then $\mathsf S A$ is itself in a natural way a tcdga (indeed, $\mathsf S$ is a symmetric monoidal endofunctor on the category of symmetric sequences).
\end{para}

\begin{thm}\label{lie0}Let $A$ be a tcdga, and $\mathsf S A$ its suspension. Let $U = \Pi_n\setminus \{\hat 0\}$. Then $\CF(U,A) \simeq C^{\mathrm{CE}}_\bullet(\mathsf S A \otimes_H \mathsf{Lie})(n)$, where the operad $\mathsf{Lie}$ is considered as a left module over itself. \end{thm}

%

%

\begin{proof}We need to carefully piece together several observations.
	\begin{enumerate}
		\item The complex of cellular cochains $\widetilde C^\bullet(\nerveud{\Pi_n},\Z)$ is isomorphic to the arity $n$ component of $\Omega \mathsf{coCom}$, the cobar construction on the cocommutative co-operad.  See \cite[Observation 6.1]{fressepartition}, as well as the \emph{Prolog} of loc.\ cit.\ for a large number of bibliographical references concerning variations of this result.
		\item The cobar construction $\Omega \mathsf{coCom}$ is quasi-isomorphic to the operadic suspension $\Lambda \mathsf{Lie}$ of the Lie operad, by Koszul duality between the commutative and Lie operads. 
		\item The Chevalley--Eilenberg chains on a twisted Lie algebra $\mathfrak g$ are given by the cofree left comodule over the cocommutative co-operad on $\mathfrak g[1]$, equipped with the Chevalley--Eilenberg differential. In particular the underlying symmetric sequence of $C^{\mathrm{CE}}_\bullet(\mathfrak g)$ is given by $\mathsf{coCom} \circ \mathfrak g[1]$. Since $\mathsf{coCom}(n)$ is the trivial representation of $\mathbb S_n$, the composition product $\mathsf{coCom} \circ \mathfrak g[1]$ is exactly the ``exponential'' considered in Remark \ref{rem2}.
		\item $\CF(U,A)$ can be written as a direct sum indexed by partitions $T \in \Pi_n$, where the summand corresponding to a partition $\{1,\ldots,n\}=\coprod_{i\in I} T_i$ is given by the tensor product $\bigotimes_{i\in I} A(T_i) \otimes \widetilde C^\bullet(\nerveud{\Pi_{T_i}})$. Indeed, $\CF$ is a sum over chains in the partition lattice, and we can decompose the complex $\CF$ according to the maximal element of the chain (which is what we denoted $T$ in the previous sentence). Comparing again with Remark \ref{rem2} we see that this is the arity $n$ component of $\mathsf{coCom} \circ (A \otimes_H \Omega \mathsf{coCom}) \simeq \mathsf{coCom} \circ (A \otimes_H
		\Lambda \mathsf{Lie})$.
	\end{enumerate}We should get the suspensions right. We have $\Lambda \mathsf{Lie}(n) = \mathsf{Lie}(n)[-n+1] \otimes \mathrm{sgn}_n$. It follows that $\mathsf{coCom} \circ (A \otimes_H
	\Lambda \mathsf{Lie}) = \mathsf{coCom} \circ (\mathsf SA \otimes_H
	\mathsf{Lie})[1] = C^{\mathrm{CE}}_\bullet(\mathsf SA \otimes_H
	\mathsf{Lie})$. We omit the verification that the differential in the complex $\CF$ can be identified with the Chevalley--Eilenberg differential (it is harder to write down than to derive).
\end{proof}

\begin{para}Let us now take $A = R\Gamma_c^\otimes(X,\EuScript F)$, in which case $\mathsf S A = R\Gamma_c^\otimes(X,\EuScript F[-1])$. Theorem \ref{lie0} combined with Theorem \ref{mainthm} gives: 
\end{para}

\begin{cor}\label{lie}Let $X$ be a paracompact and locally compact Hausdorff space, and $\EuScript F$ a bounded below complex of sheaves of $R$-modules on $X$, where $R$ is a ring of finite global dimension. Then 
	$$H_\bullet^{\mathrm{CE}}(R\Gamma_c^\otimes (X,\EuScript F[-1]) \otimes_H \mathsf{Lie}) \cong \bigoplus_{n \geq 0} H^{-\bullet}_c(F(X,n),\EuScript F^{\boxtimes n}).$$
\end{cor}

\begin{rem}\label{koszuldualityremark}
	The proof actually gives a slightly stronger chain level result. Consider on one hand the twisted dg Lie algebra $R\Gamma_c^\otimes(X,\EuScript F[-1])\otimes_H \mathsf{Lie}$, and on the other hand the twisted dg cocommutative coalgebra whose arity $n$ component is given by $C^\bullet_c(F(X,n),\EuScript F^{\boxtimes n})$. Its comultiplication is the map
	$$ C_c^\bullet(F(X,n+m),\EuScript F^{\boxtimes (n+m)}) \longrightarrow C_c^\bullet(F(X,n),\EuScript F^{\boxtimes n}) \otimes C_c^\bullet(F(X,m),\EuScript F^{\boxtimes m})$$
	given by extension by zero, noting that $F(X,n+m)$ is an open subspace of $F(X,n) \times F(X,m)$. The claim is that these algebras are dual to each other under Koszul duality.
\end{rem}

\begin{cor}\label{manifoldcor}
	Let $X$ be a paracompact and locally compact Hausdorff space. If $A$ denotes a cdga model for the compactly supported cochains on $X$ with $\Q$-coefficients, then 
	$$H_\bullet^{\mathrm{CE}}(A \otimes \mathsf S\mathsf{Lie}) \cong \bigoplus_{n \geq 0} H^{-\bullet}_c(F(X,n),\Q).$$
\end{cor}

\begin{cor}
	Let $X$ be a manifold of dimension $d$ (possibly with boundary), $\EuScript L$ the orientation sheaf on $X$. Then 
	$$H_\bullet^{\mathrm{CE}}(R\Gamma_c^\otimes (X,\EuScript L[d-1]) \otimes_H \mathsf{Lie}) \cong \bigoplus_{n \geq 0} H_{\bullet}(F(X,n),\Z).$$
\end{cor}

\begin{proof}
	The homology of any space $X$ is canonically isomorphic to the compactly supported cohomology of its dualizing complex $\mathbb D \Z$. When $X$ is a manifold, its dualizing complex is $\EuScript L[\dim X]$, the orientation sheaf shifted into degree $\dim X$. 
\end{proof}

\begin{para}Corollary \ref{lie} gives a generalization of results of Knudsen \cite{knudsenconfiguration} and H\^o \cite{hoconfiguration}. Namely, if we specialize Corollary \ref{lie} to the case that $X$ is a manifold, $\EuScript F$ is the dualizing complex of $X$ (i.e.\ the orientation sheaf of $X$ placed in cohomological degree $\dim X$) tensored with $\Q$, and we take $\mathbb S_n$-invariants in each arity, then we recover exactly the main theorem of \cite{knudsenconfiguration}. The main result of \cite{hoconfiguration} is the analogous statement in the \'etale cohomology of an algebraic variety $X$ with $\Q_\ell$-coefficients, and the sheaf-theoretic methods used here can treat this case as well, as indicated in Remark \ref{algebraic}. Some statements closely related to Corollary \ref{lie} can be found in the literature already. In an unpublished preprint, Getzler \cite[p.4]{nilpotent} states an isomorphism $H_\bullet^{\mathrm{CE}}(\Omega_c^\bullet(M) \otimes \mathsf{Lie}) \cong \bigoplus_{n \geq 0} H^\bullet_c(F(M,n),\mathbf R)[-n] \otimes \mathrm{sgn}_n$, where $M$ is an oriented manifold and $\Omega_c^\bullet(M)$ is the compactly supported de Rham complex. (The formula in loc.\ cit.\ is stated without compact supports on both sides, but the above formula is presumably what is intended.) 
	
	One can also extract a version of Corollary \ref{lie} from computations in Goodwillie calculus. If $X$ is a based space, let $\bar F(X,n) = X^{\wedge n}/D_n$, where $D_n$ is the ``big diagonal'' inside the smash product. Now note that if $X$ is the one-point compactification of a space $Y$, then $\bar F(X,n)$ is the one-point compactification of $F(Y,n)$. In particular, the reduced cohomology of $\bar F(X,n)$ is the compactly supported cohomology of $F(Y,n)$. We now need the following inputs from Goodwillie calculus, see \cite{aronesnaith,chingbarconstructions} and in particular \cite[Example 17.28]{aroneching}:
	 \begin{itemize}
	 	\item Let $F \colon \mathsf{Top}_\ast \to \mathsf{Top}_\ast$ be a pointed homotopy functor. Then its derivatives $\partial_\ast F$ form a left module over $\partial_\ast I_{\mathsf{Top}_\ast}$, the ``spectral Lie operad''.
	 	\item If we instead consider $G \colon \mathsf{Top}_\ast \to \mathsf{Spectra}$, then $\partial_\ast G$ is a left module over $\partial_\ast(\Sigma^\infty \Omega^\infty)$, a spectral version of the commutative operad.
	 	\item If $G = \Sigma^\infty F$, then $\partial_\ast G$ and $\partial_\ast F$ are Koszul dual to each other.  
	 	\item If $F = \mathrm{map}(X,-)$, then its $n$th derivative is given by $X \wedge \partial_n I_{\mathsf{Top}_\ast}$.
	 	\item Let $X$ be a finite complex. The $n$th derivative of $\Sigma^\infty\mathrm{map}(X,-)$ is given by $\mathbb D \bar F(X,n)$, where $\mathbb D$ denotes the Spanier--Whitehead dual .
	 \end{itemize}
	 Putting it all together, we see that we obtain a ``spectral'' version of the Koszul duality statement of Remark \ref{koszuldualityremark}. Taking cohomology recovers Remark \ref{koszuldualityremark}, but only in the case of the constant sheaf $\Z$. The statement obtained from Goodwillie calculus is neither more nor less general than the one here: the spectral version allows one to work with an arbitrary cohomology theory, but the sheafy version allows e.g.\ to plug in the dualizing complex of $X$. It seems likely that there exists a six-functors formalism for sheaves of spectra over spaces and that the methods of this paper would work equally well in such a setting to prove a statement which specializes to both formulae, but no such formalism exists in the literature.   
\end{para}

\begin{rem}
	Here is a simple example of the kind of result one could prove using a hypothetical formalism of six functors for sheaves of spectra. If we take for $\EuScript F$ in Corollary \ref{lie} the dualizing complex then on the right hand side we should get the stable homotopy types of the configuration spaces of points on $X$. The theorem would therefore answer the question of how much more information than the homotopy type of $X$ one needs to determine the stable homotopy type of the configuration space of points on $X$: one needs also to know the dualizing complex. For example, when $X$ is a manifold, then dualizing complex should be the parametrized Thom spectrum given by the tangent bundle (this would be a form of Atiyah duality), and since the stable tangent bundle is a homotopy invariant of compact manifolds one would deduce that the stable homotopy type of configuration spaces of points on a compact manifold $M$ depends only on the homotopy type of $M$, recovering a theorem of Aouina--Klein \cite{aouinaklein}. Knudsen \cite[Theorem C]{knudsenhigher} gave a proof of the Aouina--Klein result in a slightly stronger form (for example, Knudsen's result is equivariant) using factorization homology, which is at least morally the ``same'' proof as the one suggested in this remark (in the manifold case).
\end{rem}

\begin{para}Let us explain how to recover the results of Knudsen and H\^o from Corollary \ref{lie}. Suppose we work over $\Q$. Then there is an \emph{exact} symmetric monoidal functor 
$$ \text{(symmetric sequences of chain complexes)} \stackrel{\mathrm{inv}}\longrightarrow \text{(chain complexes with an extra $\mathbf N$-grading)} $$
given by taking $\mathbb S_n$-invariants in each arity. If $\mathfrak g$ is a twisted dg Lie algebra, then $\operatorname{inv} \mathfrak g$ is a Lie algebra in the category of $\mathbf N$-graded dg $\Q$-vector spaces, and since $\mathrm{inv}$ is exact it commutes with taking Lie algebra cohomology: $\operatorname{inv} H^{\mathrm{CE}}_\bullet(\mathfrak g) \cong H^{\mathrm{CE}}_\bullet(\operatorname{inv} \mathfrak g)$. We get in particular:
\end{para}
\begin{prop}
	Let $X$ be a locally compact Hausdorff space, and let $\mathbb D \Q$ be the dualizing complex of $X$ with rational coeffcients. Then
	$$ H_\bullet^{\mathrm{CE}}(\operatorname{inv} R\Gamma_c^\otimes(X,\mathbb D\Q[-1]) \otimes_H \mathsf{Lie}) \cong \bigoplus_{n \geq 0} H_\bullet(F(X,n)/\mathbb S_n,\Q ).$$
\end{prop}
\begin{para}
	Now the main result of \cite{knudsenconfiguration} is the description of an explicit $\mathbf N$-graded Lie algebra over $\Q$ associated to a manifold $X$, whose homology is given by the direct sum $\bigoplus_{n \geq 0} H_\bullet(F(X,n)/\mathbb S_n,\Q )$. What we must show is therefore that $\operatorname{inv} R\Gamma_c^\otimes(X,\mathbb D\Q[-1]) \otimes_H \mathsf{Lie}$ is isomorphic to this $\mathbf N$-graded Lie algebra when $X$ is a manifold. We will need two observations:\begin{itemize}
	\item Suppose the complex $\EuScript F$ consists of a single sheaf placed in even degree. Then the cohomology of $R\Gamma_c^\otimes(X,\EuScript F)(n)$ consists only of the trivial representation of $\mathbb S_n$. If instead $\EuScript F$ is concentrated in odd degree, then the cohomology of $R\Gamma_c^\otimes(X,\EuScript F)(n)$ transforms via the sign representation of $\mathbb S_n$.
	\item The trivial representation occurs in $\mathsf{Lie}(n)$ only for $n=1$, where $\mathsf{Lie}(1) = \Q$ is the trivial representation. The sign representation of $\mathbb S_n$ occurs only for $n=1$ and $n=2$, where $\mathsf{Lie}(1)$ and $\mathsf{Lie}(2)$ are both given by a copy of the sign representation. 
\end{itemize}
So let $X$ be a $d$-dimensional manifold. We see from the above considerations that if $d$ is odd then $\operatorname{inv} R\Gamma_c^\otimes(X,\EuScript L[d-1]) \otimes_H \mathsf{Lie}$ is concentrated in degree $1$, where it is given by $H^\bullet_c(X,\EuScript L[d-1]) \cong H_{\bullet+1}(X,\Q)$ (with identically zero Lie bracket).  If $d$ is even then  $\operatorname{inv} R\Gamma_c^\otimes(X,\EuScript L[d-1]) \otimes_H \mathsf{Lie}$ is concentrated in degrees $1$ and $2$. Its degree $1$ component is $C^\bullet_c(X,\EuScript L[d-1])$, and its degree $2$ component is $C^\bullet_c(X,\EuScript L[d-1]^{\otimes 2})\cong C^{\bullet+2d-2}_c(X,\Q)$, which has an evident 2-step nilpotent Lie bracket given by cup product. This Lie algebra is moreover formal, by the argument of \cite[Remark 7.7]{knudsenconfiguration}, and hence we can take the components to be given by the cohomologies $H^\bullet_c(X,\EuScript L)$ and $H^\bullet_c(X,\Q)$ (appropriately shifted). We have recovered exactly the Lie algebra defined by Knudsen.
\end{para}

\bibliographystyle{alpha}
\bibliography{../database}

\end{document}